\documentclass[11pt, a4paper]{article}
\usepackage[utf8]{inputenc}
\usepackage{mathtools}
\usepackage{amsmath}
\usepackage{amsthm}
\usepackage[linesnumbered,ruled,vlined]{algorithm2e}
\usepackage{algorithmic}
\RestyleAlgo{ruled}
\SetKwInput{KwInput}{Input}
\SetKwInput{KwOutput}{Output}
\usepackage{algorithmic}
\usepackage{amssymb}
\usepackage{enumitem}
\usepackage{parskip}
\usepackage{caption}
\usepackage{subcaption}
\usepackage{epsfig}
\usepackage[colorlinks, citecolor=hanblue,linkcolor=mordantred19,urlcolor=darkgreen]{hyperref}
\usepackage{stackengine}
\mathtoolsset{showonlyrefs}
\usepackage{todonotes}
\usepackage{xcolor}
\usepackage[a4paper, margin=2.5cm]{geometry}
\usepackage{comment}
\usepackage{bbm}
\usepackage{mathrsfs}
\usepackage{float}
\usepackage{tikz, pgfplots}
\usetikzlibrary{decorations.pathreplacing,calligraphy}
\pgfplotsset{compat=1.18}
\usepackage[export]{adjustbox}
\DeclareMathOperator*{\argmax}{arg\,max}
\DeclareMathOperator*{\argmin}{arg\,min}
\usepackage{multirow}
\usepackage{setspace}

\newcommand{\mres}{\mathbin{\vrule height 1.4ex depth 0pt width
0.13ex\vrule height 0.13ex depth 0pt width 1.0ex}}
\usepackage{color}
\definecolor{hanblue}{rgb}{0.27, 0.42, 0.81}
\definecolor{mordantred19}{rgb}{0.68, 0.05, 0.0}
\definecolor{darkgreen}{rgb}{0.0, 0.38, 0.12}
\definecolor{red}{rgb}{0.8, 0.0, 0.0}
\definecolor{green}{rgb}{0.0, 0.5, 0.0}

\newcommand{\B}{\mathcal{B}}

\newcommand{\weakstar}{\stackrel{*}{\rightharpoonup}}
\newcommand{\st}{\,|\,}

\newcommand{\e}{\varepsilon}

\DeclarePairedDelimiter{\mnorm}  {\lVert}{\rVert_{\mathcal{M}}}
\DeclarePairedDelimiter{\cnorm}  {\lVert}{\rVert_{\mathcal{C}}}

\DeclarePairedDelimiter{\omegaynorm}  {\lVert}{\rVert_{\mathcal{C}(\Omega;Y)}}

\DeclareMathOperator{\sign}{sign}

\DeclareMathOperator{\supp}{supp}

\DeclarePairedDelimiter{\seq}()

\usepackage{mathtools}
\usepackage{hyperref}
\mathtoolsset{showonlyrefs}

\newcommand{\R}{\mathbb{R}}
\newcommand{\A}{\mathcal{A}}

\newcommand{\C}{\mathcal{C}}
\newcommand{\E}{\mathcal{E}}
\newcommand{\D}{\mathcal{D}}
\newcommand{\J}{\mathcal{J}}

\newcommand{\N}{\mathbb{N}}
\renewcommand{\L}{\mathcal{L}}
\newcommand{\cC}{\mathcal{C}}

\newcommand{\dd}{\, \mathrm{d}}

\newcommand{\Moc}{\mathcal{M}(\Omega)}
\newcommand{\Mm}{\mathcal{M}}
\newcommand{\M}[1]{\mathcal{M}(#1)}
\DeclarePairedDelimiterX\set[1]\{\}{#1}
\DeclarePairedDelimiterX\dual[2]{\langle}{\rangle}{#1,#2}

\newcommand\restr[2]{{
  \left.\kern-\nulldelimiterspace 
  #1 
  \vphantom{\big|} 
  \right|_{#2} 
  }}
\numberwithin{equation}{section}
\theoremstyle{plain}
\newtheorem{theorem}{Theorem}[section]

\newtheorem{lemma}[theorem]{Lemma}
\newtheorem{proposition}[theorem]{Proposition}
\newtheorem{corollary}[theorem]{Corollary}

\newtheorem{assumption}{Assumption}

\newtheorem{remark}[theorem]{Remark}
\theoremstyle{definition}
\newtheorem{definition}[theorem]{Definition}
\definecolor{hanblue}{rgb}{0.27, 0.42, 0.81}
\usepackage{scalerel}

\title{Lazifying point insertion algorithms in spaces of measures}
\author{Arsen Hnatiuk$^{\ast}$, Daniel Walter\thanks{Institut f\"ur Mathematik, Humboldt-Universit\"at zu Berlin, 10117 Berlin, Germany \newline (\texttt{arsen.hnatiuk@hu-berlin.de, daniel.walter@hu-berlin.de})}}
\date{}

\begin{document}
\maketitle

\begin{abstract}
Greedy point insertion algorithms have emerged as an attractive tool for the solution of minimization problems over the space of Radon measures. Conceptually, these methods can be split into two phases: first, the computation of a new candidate point via maximizing a continuous function over the spatial domain, and second, updating the weights and/or support points of all Dirac-Deltas forming the iterate. Under additional structural assumptions on the problem, full resolution of the subproblems in both steps guarantees an asymptotic linear rate of convergence for pure coefficient updates, or finite step convergence, if, in addition, the position of all Dirac-Deltas is optimized. In the present paper, we lazify point insertion algorithms and allow for the inexact solution of both subproblems based on computable error measures, while provably retaining improved theoretical convergence guarantees. As a specific example, we present a new method with a quadratic rate of convergence based on Newton steps for the weight-position pairs, which we globalize by point-insertion as well as clustering steps.  
\end{abstract}



\section{Introduction} \label{sec:introduc}
Given a compact set $\Omega \subset \R^d$ as well as a convex fidelity measure $F$, we are interested in numerical algorithms for minimization problems on the space $\Mm$ of Radon measures on $\Omega$,
\begin{equation}
    \tag{$\mathcal{P}$}
    \label{eq: problem setting}
    \min_{u\in\mathcal{M}}J(u)= \left \lbrack F(Ku)+\alpha\mnorm{u} \right \rbrack, \quad \text{where} \quad K\mu= \int_\Omega \kappa(x)~\dd \mu(x)
\end{equation}
and $\kappa \colon \Omega \to Y $ denotes a kernel function mapping to a Hilbert space $Y$. The image of the latter can be interpreted as a potentially continuous dictionary of elements in $Y$, which is indexed by the set $\Omega$. In particular, this ansatz allows for modeling linear combinations of atoms in the dictionary via \textit{sparse measures},
\begin{equation*}
    Ku= \sum^N_{j=1} \lambda^j \kappa(x^j), \quad \text{where} \quad u=\mathcal{U}(\mathbf{x}, \lambda)= \sum^N_{j=1}\lambda^j\delta_{x^j}
\end{equation*}
is the associated weighted sum of Dirac-Delta functionals. By incorporating the Radon norm $\mnorm{\cdot}$ as a regularizer in \eqref{eq: problem setting}, we encourage the existence of minimizers $\bar{u}$ exhibiting this desired structural property, i.e. 
\begin{equation}\label{eq: sparse solution}
   \bar{u}=\mathcal{U}(\bar{\mathbf{x}}, \bar{\lambda})= \sum^{\bar{N}}_{j=1} \bar{\lambda}^j \delta_{\bar{x}^j}, \quad \text{where} \quad (\bar{\mathbf{x}}, \bar{\lambda}) \in \argmin_{\mathbf{x}\in \Omega^{\bar{N}},\lambda \in \R^{\bar{N}}} \left\lbrack F(K \mathcal{U}(\mathbf{x},\lambda))+ \alpha |\lambda|_{\ell_1} \right \rbrack,
\end{equation}
which, e.g., follows from convex representer theorems, \cite{Brediessparse, Chambolle}, if $Y$ is finite-dimensional, or which can often be deduced from properties of the optimal dual variable associated with \eqref{eq: problem setting} via its first-order optimality conditions. Noting that  \eqref{eq: problem setting} is convex, albeit at the cost of working in the infinite-dimensional space $\Mm$, this approach has received tremendous attention in a variety of fields, ranging from super-resolution approaches in signal denoising, to optimal control and related inverse problems as well as machine learning applications and system identification. For a non-exhaustive overview of related work, we refer, e.g., to \cite{PieperKonstantin2021Lcoa, Flinth25,Huynh,Duvalsource} and the references mentioned therein.

Naturally, these observations suggest to exploit the expected, finite-dimensional parametrization of the sought-after solution in the design of numerical methods for Problem \eqref{eq: problem setting}. In this regard, our interest lies in \text{greedy point insertion} algorithms, such as the Primal-Dual-Active Point method of \cite{PieperKonstantin2021Lcoa}, or the Sliding Frank-Wolfe ansatz of \cite{DenoyelleQuentin2020TsFa}. Loosely speaking, these alternate between the update of a sparse iterate $u_k$ and that of a finite, ordered set $\mathcal{A}_k$, the \textit{active set}, comprising its support points, i.e. 
\begin{equation*}
    u_k= \mathcal{U}(\mathbf{x}_k, \lambda_k)=\sum^{N_k}_{j=1} \lambda_k^j \delta_{x_k^j} , \quad \mathcal{A}_k= \{x^j_k\}^{N_k}_{j=1}
\end{equation*}
for some $(\mathbf{x}_k, \lambda_k) \in \Omega^{N_k} \times \R^{N_k} $. More in detail, it greedily adds points,  
\begin{equation*}
    \mathcal{A}_{k,+}= \mathcal{A}_k \cup\{\widehat{x}_k\}, \quad \widehat{x}_k \in \argmax_{x \in \Omega} |p_k| , \quad p_k=-K_* \nabla F(Ku_k),
\end{equation*}
and then either performs convex \textit{coefficient minimization}
\begin{equation*}
    u_{k+1}=\mathcal{U}(\mathbf{x}_{k,+}, \lambda_{k,+}), \quad \text{where} \quad \lambda_{k,+} \in \argmin_{\lambda \in \R^{\# \mathcal{A}_{k,+}}} J(\mathcal{U}(\mathbf{x}^{k,+}, \lambda))
\end{equation*}
and the fixed positions $\mathbf{x}^{k,+}$ correspond to the elements of $\mathcal{A}_{k,+}$, or \textit{sliding}, i.e. optimizing both coefficients and positions 
\begin{equation*}
    u_{k+1}=\mathcal{U}(\mathbf{x}_{k,+}, \lambda_{k,+}), \quad \text{where} \quad (\mathbf{x}_{k,+}, \lambda_{k,+}) \in \argmin_{\mathbf{x}\in \Omega^{\# \mathcal{A}_{k,+}},\lambda \in \R^{\# \mathcal{A}_{k,+}}} \left\lbrack F(K \mathcal{U}(\mathbf{x},\lambda))+ \alpha |\lambda|_{\ell_1} \right \rbrack . 
\end{equation*}
Afterwards, $\mathcal{A}_{k,+}$ can be pruned, removing all elements that were assigned zero coefficients.

Conceptually, these methods can be interpreted as accelerated variants of a \textit{generalized conditional gradient method} (\texttt{GCG}), \cite{Mine,BrediesKristian2009Agcg,KunischGCG},
\begin{equation*}
    u_{k,+}=(1-\eta_k)u_k+ \eta_k v_k, \quad v_k\in\argmin_{v\in\Mm}\left[-\langle p_k,v\rangle+\alpha\mnorm{v}\right], \quad \eta_k \in [0,1],
\end{equation*}
applied to the surrogate problem
\begin{equation*}
    \min_{u\in\mathcal{M}(\Omega)}J(u) \quad \text{s.t.} \quad \mnorm{u} \leq M,
\end{equation*}
where $M>0$ is a large enough constant, noting that a suitable direction $v_k$ can be computed from $\widehat{x}_k$, see Section \ref{section:lazification}. While \texttt{GCG} is known to converge globally at a sublinear rate, a property which is also passed on to its accelerated variants, the latter exhibit a substantially improved asymptotic convergence behavior, provided that the optimal solution to Problem \eqref{eq: problem setting} is of the form \eqref{eq: sparse solution} and the associated dual variable $\bar{p}=-K_* \nabla F(K \bar{u})$ in Problem \eqref{eq: problem setting} satisfies additional strict complementarity and non-degeneracy assumptions on its curvature, in particular
\begin{equation*}
    \{\,x \in \Omega\;|\;|\bar{p}(x)|=\alpha\,\}= \{\bar{x}^j\}^{\bar{N}}_{j=1}, \quad -\sign(\bar{\lambda}^j)\nabla^2 \bar{p}(\bar{x}^j) \geq_L \theta \operatorname{Id}, \quad j=1, \dots, \bar{N}
\end{equation*}
for some $\theta>0$. However, from a practical perspective, all of these desirable properties, i.e. sparse iterates and fast convergence, are achieved at the cost of computationally expensive substeps. First and foremost, updating $\mathcal{A}_k$ requires the global maximization of the generally nonconcave function~$|p_k|$. Similarly, sliding leads to a nonconvex, nonsmooth minimization problem. Second, while the coefficient minimization problem is convex, the theoretical results rely on its exact resolution, raising the question whether these can still be ensured in practice, where inexactness is unavoidable.

\paragraph{Contribution} \label{subsec:contribution}

In the present paper, we aim to alleviate the computational complexity associated with greedy point insertion while maintaining the improved convergence behavior of its accelerated variants. For this purpose, we consider \textit{lazy} updates $\widehat{x}_k$ of the active set $\mathcal{A}_k$, as well as a relaxation of the weight-position update problems. In this context, lazy updates, in contrast to inexact or approximate maximization, do not require knowledge of the suboptimality 
\begin{equation*}
  |p_k(\widehat{x}_k)|-\max_{x\in \Omega} |p_k(x)|,  
\end{equation*}
but merely assume that $|p_k(\widehat{x}_k)|$ is large enough, quantified by an adaptive tolerance. While updates of the latter still rely on exact maximization of $|p_k|$, the lazy approach greatly reduces their number, leading to significant speed-ups.

Our contributions are threefold:
\begin{itemize}
    \item Similar to earlier approaches, we build upon the interpretation of greedy point insertion as an acceleration of \texttt{GCG}. For this purpose, we introduce a lazy variant of the latter (\texttt{LGCG}), Algorithm \ref{alg: sublinear parameter-free}, based on the template provided by \cite{BraunGábor2016LCGA} and prove its global, sublinear convergence, see Theorem \ref{theorem: sublinear}. As for exact updates, this result carries over to its accelerated versions and guarantees that these reach a neighborhood of the minimizer in which faster convergence rates can be proven. 
    \item We then turn to a lazified version (\texttt{LPDAP}) of $\texttt{PDAP}$, Algorithm \ref{alg: local-global}, and prove its asymptotic, linear convergence in Theorem \ref{thm:local linear}. From a practical perspective, the new algorithm compares the descent achieved by \texttt{LGCG} steps with a local update mechanism, Algorithm \ref{alg: LSIStep}, reminiscent of the theoretical construction in \cite{PieperKonstantin2021Lcoa}. The better of both is then refined by an inexact coefficient update, which is controlled by a cheaply computable error measure. Our analysis critically relies on the clustering of the active set $\A_k$ around the support of the minimizer. In the absence of exact coefficient minimization, this is achieved by incorporating \textit{drop steps}, see Algorithm \ref{alg: drop step}, which provably remove points far away from the optimal support.
    \item Finally, we combine the \texttt{LGCG} approach with the sliding Frank-Wolfe philosophy, \cite{DenoyelleQuentin2020TsFa}. Exploiting global \texttt{LGCG} convergence, we replace the exact solution of the weight-position minimization problem by running a Newton-like method and interpret \texttt{LGCG} as a globalization approach. Regularly comparing the Newton progress to the per-iteration descent promised by \texttt{LGCG} as well as incorporating clustering steps, the proposed Algorithm \ref{alg: local merging} (\texttt{NLGCG}) eventually identifies the correct number of support points and always accepts the Newton step. Hence, new and improved convergence results on the infinite-dimensional level follow from classical finite-dimensional arguments, see Theorem \ref{thm:quadconvergence}.      
\end{itemize}
Our theoretical results are confirmed by numerical experiments, which, while simple, emphasize the main benefits of the lazy paradigm.  

\paragraph{Related work \& limitations} \label{subsec:related}
Conditional gradient methods with inexact linear minimization have been considered, e.g., in \cite{juliensvm,pmlr-v28-jaggi13, dunnopen}. A transfer to \texttt{GCG}-like methods can be found in \cite{yusparsity}. In contrast, we are not aware of comparable extensions of the lazy paradigm despite the significant interest it has attracted, \cite{BraunGábor2016LCGA,Braunblending,LanPokutta2017}.  

To the best of our knowledge, (accelerated) \texttt{GCG}-like methods for problems of the form were first considered in \cite{BrediesInverseroblems}, albeit without improved convergence guarantees beyond the global, sublinear rate. However, we also mention the intricate connection to the classical Federov-Wynn algorithm in the context of optimal design of experiments, \cite{fedorov, wynn}. The subsequent works \cite{Flinth,PieperKonstantin2021Lcoa} provide first asymptotic linear rates for acceleration by exact coefficient minimization, given the aforementioned structural assumptions on the optimal dual variable. In \cite{wachsmuthwalter}, the latter are related to no-gap second-order conditions and local quadratic growth w.r.t. the Kantorovich-Rubinstein norm. For finite-dimensional $Y$, the manuscript \cite{Flinth} exploits the connection between accelerated \texttt{GCG} and exchange-type methods, \cite{hettich}, applied to the predual problem, which is constituted by a semi-infinite program as pointed out in \cite{Eftekhari}. Variants allowing for point moving also go back to \cite{BrediesInverseroblems} and encompass, e.g., the alternating descent algorithm, \cite{Schiebinger}, the hybrid approach in \cite{Flinth}, or the sliding Frank-Wolfe ansatz, \cite{DenoyelleQuentin2020TsFa}, all of which provide finite-step convergence. An extension for inverse problems with Poisson noise, albeit without fast convergence results, is discussed in~\cite{Calatroni25}. 

Common to all of these approaches is that the proofs of improved convergence behavior critically depend on the exact maximization of the dual variable as well as the computation of critical points in the arising subproblems, leveraging information provided by the respective optimality conditions. To the best of our knowledge, the only ansatz relaxing these requirements is found in \cite{Flinth25}, where the authors replace $\Omega$ by a finite, adaptively refined grid. However, in contrast to the present work, linear convergence guarantees still require exact coefficient minimization. For the treatment of inexactness in semi-infinite programming, we refer, e.g., to \cite{oustry}, which considers a blackbox oracle guaranteeing a multiplicative error estimate, or \cite{mitsos17} as well as the related literature discussed therein.   

For completeness, we also mention philosophically different approaches based on overparametrization, \cite{ChizatLénaïc2022Soom}, trading off small support sizes for simple closed-form update steps, as well as prox-like methods, \cite{Tuomo1,Tuomo2}, which are able to deal with inexactness but so far lack improved convergence results.

The \texttt{LPDAP} method presented in this manuscript is directly inspired by the constructions employed in the linear convergence proofs of \cite{PieperKonstantin2021Lcoa}. Similarly, \texttt{NLGCG} is closely related to the hybrid approach of \cite{Flinth} but does neither require the computation of all local maximizers of $|p_k|$ to update the active set, nor exact coefficient minimization in order to achieve improved convergence rates.

While promising, the proposed lazy ansatz is of course not without limitations. First and foremost, lazy \texttt{GCG} steps do not fully remove the need for exact maximization of the dual variable, since the latter is occasionally required to update the lazy threshold. On the one hand, for \texttt{LPDAP}, our experiments suggest that these exact updates happen in regular intervals, but their overall number is small compared to the original method from \cite{PieperKonstantin2021Lcoa}, leading to a significant speed-up in practice. On the other hand, for \texttt{NLGCG}, exact updates predominantly happen in the asymptotic regime, i.e., once the correct number of support points is identified and the algorithm exhibits the local quadratic convergence behavior of Newton's method. In this case, our analysis suggests that the global maximizers of $|p_k|$ lie in the vicinity of the active set $\A_k$, which alleviates the exact update tremendously by providing a good warmstart.

Second, the presented algorithms heavily rely on hyperparameters that estimate problem-specific constants such as Lipschitz and curvature parameters, as well as the separation distance between optimal points. However, we emphasize that a parameter-free, adaptive version can be analyzed mutatis mutandis at the cost of additional technicalities in the proofs and computational effort to estimate relevant quantities on the fly. While we do not pursue this route in this paper in order to strike a balance between readability and technical details, the adaptive algorithm will be presented in a follow-up paper, together with more challenging numerical experiments.

\paragraph{Outline} \label{subsec:outline}

This paper is structured as follows. After introducing the relevant notation in Section~\ref{sec:notation}, we state the problem setting, the necessary assumptions, and use them to derive immediate properties of primal and dual variables in Section~\ref{sec:optimality}. In  Sections~\ref{section:lazification}, \ref{section:lpdap}, and \ref{section: Newton}, we present the \texttt{LGCG}, \texttt{LPDAP}, and \texttt{NLGCG} algorithms, respectively, and derive their convergence properties. Lastly, in Section~\ref{section: numerics}, we discuss a numerical implementation of the algorithms in the settings of PDE-constrained optimization and signal processing. We analyze the observed convergence behavior and compare it to the theory.

\section{Notation} \label{sec:notation}
Throughout the following, let~$\Omega \subset \R^d $,~$d \geq 1$, be a compact set and let $Y$ be a Hilbert space with inner product $(\cdot,\cdot)_Y$. The associated induced norm on $Y$ is denoted by $\Vert\cdot\Vert_Y= \sqrt{(\cdot,\cdot)_Y}$, while $\Vert\cdot\Vert$ refers to the euclidean norm on $\R^n$, $n >1$.

For a set $\Omega'\subseteq\Omega$ let $\C(\bar \Omega')$ and $\C^{0,\nu}(\bar \Omega')$ denote the space of continuous and $\nu$-Hölder continuous functions on $\bar \Omega'$. We equip $\C(\bar \Omega')$ with the canonical norm $\|\cdot\|_{\C(\bar \Omega')}$. Moreover, if $\Omega'$ is open, we denote the spaces of $n$-times continuously differentiable functions on $\Omega'$ whose derivatives can be continuously extended to $\bar{\Omega}'$ by $\C^n(\bar \Omega')$. The spaces $\C^{n,\nu}(\bar \Omega')$ of functions with $\nu$-Hölder continuous $n$-th derivative are defined analogously. In both cases, the respective spaces are equipped with the canonical norm. Mutatis mutandis, we define the corresponding spaces $\C(\bar \Omega';H)$, $\C^{0,\nu}(\bar \Omega';H)$, $\C^{n}(\bar \Omega';H)$, and $\C^{n,\nu}(\bar \Omega';H)$ for functions taking values in a separable Hilbert space $H$.

Abbreviating $\C\coloneqq\C(\Omega)$ and $\|\cdot\|_{\C} \coloneqq\|\cdot\|_{\C( \Omega)}$, we introduce the space of Radon measures $\Mm$ on $\Omega$ as the topological dual space of $\C$ with duality pairing $\dual{\cdot}{\cdot}$ and induced norm
\begin{equation}
    \dual{\varphi}{u}=\int_\Omega \varphi(x)~\mathrm{d}u(x)\quad\text{for all}\quad\varphi\in\C,  u\in\Mm, \quad  \mnorm{u}= \sup_{\|\varphi\|_{\C}\leq 1} \dual{\varphi}{u},
\end{equation}
making it a Banach space. The \textit{support} of a measure $u \in \Mm$ is denoted by $\supp(u)$. Given $x \in \Omega$, $\delta_x$ denotes the associated Dirac-Delta functional, i.e. $\dual{\varphi}{\delta_x}=\varphi(x)$ for all $\varphi \in \C$. Throughout this paper, we call $u \in \Mm$ \textit{sparse} if there is a finite, ordered set of distinct points $\A_u=\{x^j\}^N_{j=1}$  as well as nonzero coefficients $\{\lambda^j\}^N_{j=1}$  such that
\begin{equation*}
    u= \sum^N_{j=1} \lambda^j \delta_{x^j}, \quad \text{where} \quad \A_u= \supp(u), \quad \mnorm{u}= |\lambda|_{\ell_1}.
\end{equation*}
For a finite set $\A$ denote by $\M{\A}$ the linear subspace of sparse measures $u$ with $\A_u \subset \A$.
 
Given a Borel set $\Omega' \subset \Omega $, the restriction of $u \in \Mm$ to $\Omega'$ is denoted by $u\mres\Omega'= u(\cdot \cap \Omega')$.

Throughout this paper, sequences are written as indexed elements inside parentheses $\seq{\cdot}$, where the index is not further specified unless necessary. We also write $( \cdot )_+ \coloneqq \max\{ \cdot ,0\}$. Finally, $B_R(\bar{x})$ denotes the (open) ball of radius $R>0$ around $\bar{x} \in \R^d$.

\section{A primer on sparse minimization problems} \label{sec:optimality}

In the following two sections, we collect pertinent results on minimization problems of the form~\eqref{eq: problem setting}
\begin{equation}
    \min_{u\in\Mm}J(u)= \left \lbrack F(Ku)+\alpha\mnorm{u} \right \rbrack, \quad \text{where} \quad Ku=\int_\Omega \kappa(x)~\mathrm{d}u(x).
\end{equation}
The following standing assumptions are made throughout the paper:
\begin{assumption} \label{ass:functions}
    We assume that:
    \begin{enumerate}[label=\textbf{A\arabic*}]
       \item The kernel~$\kappa \colon \Omega \to Y$ satisfies~$\kappa \in \mathcal{C}^{0,\nu}(\Omega;Y)$ for some~$\nu >0$.
        \item The diligence measure $F:Y\rightarrow\mathbb{R}_+$ is strictly convex and continuously Fr\'echet differentiable with gradient~$\nabla F \colon Y \to Y $.
        \item There is an $L_{\nabla F}>0$ such that
        \begin{align*}
            \|\nabla F(y_1)-\nabla F(y_2)\|_Y \leq L_{\nabla F} \|y_1-y_2\|_Y \quad \text{for all} \quad y_1, y_2 \in Y.
        \end{align*}
    \end{enumerate}
    
\end{assumption}

The following lemma follows immediately, cf. also \cite[Lemma 3.2]{wachsmuthwalter}.

\begin{lemma} \label{lem:propofK}
    Let Assumption~\ref{ass:functions} hold. Then the operator~$K \in \mathcal{L}(\Mm;Y)$ is weak*-to-strong continuous. Moreover, we have
    \begin{align*}
    (Ku, y)_Y = \langle K_* y, u  \rangle \quad \text{for all}\quad u \in \Mm,~y \in Y    ,
    \end{align*}
    where~$K_* \in \L(Y;\C)$ satisfies
    \begin{align*}
        \lbrack K_* y(x) \rbrack(x)= (\kappa(x),  y)_Y \quad \text{for all}\quad x \in \Omega, y \in Y .
    \end{align*}
\end{lemma}
Note that Lemma~\ref{lem:propofK}, together with the differentiability requirements in Assumption~\ref{ass:functions}, guarantees that~$f= F \circ K$ is Fr\'echet-differentiable and the directional derivative in a direction~$\delta u \in \Mm$ satisfies   
\begin{align*}
    f'(u)(\delta u)=\langle K_* \nabla F(K u), \delta u  \rangle \quad \text{for all} \quad u \in \Mm.
\end{align*}

Furthermore, the weak* lower semicontinuity of $\mnorm{\cdot}$ also implies that $J$ is weak* lower semicontinuous.

\paragraph{Existence of minimizers \& first order optimality conditions} \label{para:existenceandopt}

Lemma \ref{lem:propofK}, together with Assumption~\ref{ass:functions}, allows to conclude the existence of minimizers to~\eqref{eq: problem setting}, as well as the derivation of usable first-order necessary and sufficient conditions. While these are crucial for the remainder of the paper, their proofs are standard and are thus omitted for the sake of brevity.  
\begin{proposition} \label{prop:existence}
    There exists at least one solution~$\bar{u} \in \Mm$ to Problem~\eqref{eq: problem setting} and for two minimizers~$\Bar{u}_1, \bar{u}_2 \in \Mm$ there holds~$K\bar{u}_1=K\bar{u}_2$. Moreover, for every~$u \in \Mm$, the sublevel set
    \begin{align*}
        E_{J}(u)=\left\{ v \in \Mm\;|\;J(v)\leq J(u) \right\}
    \end{align*}
    is weak*-compact.
\end{proposition}

In view of this, we define the \textit{residual}
\begin{align*}
 r_J(u)\coloneqq J(u)- \min_{v \in \Mm}J(v)   
\end{align*}
of a measure $u\in\Mm$. Furthermore, we refer to~$\bar{y}=K \bar{u}$ as the unique \textit{optimal observation} associated with Problem~\eqref{eq: problem setting}. For~$u \in \Mm$, we further define the associated \textit{dual variable} as
\begin{align*}
    p_u=-K_* \nabla F(K u) \in \C .
\end{align*}

\begin{proposition} \label{prop:firstorder}
Let~$\bar{u} \in \Mm$ with~$J(\bar{u})<\infty$ be given and set~$\bar{p}=p_{\bar{u}}$. Then~$\bar{u}$ is a minimizer of Problem~\eqref{eq: problem setting} if and only if $\cnorm{\bar{p}} \leq \alpha$
and one of the following (equivalent) conditions holds:
\begin{itemize}
    \item There holds~$\langle \bar{p}, \bar{u} \rangle= \alpha \mnorm{\bar{u}}$.
    \item The Jordan-decomposition~$\bar{u}=\bar{u}_+ -\bar{u}_{-}$ satisfies
    \begin{align*}
        \supp(\bar{u}_{\pm}) \subset \left\{ x \in \Omega\;|\;\bar{p}(x)=\pm \alpha \right\}.
    \end{align*}
\end{itemize}
\end{proposition}
As a consequence of Proposition~\ref{prop:existence}, the optimal dual variable $\bar{p}=p_{\bar{u}}$ associated with Problem~\eqref{eq: problem setting} is unique as well.

\paragraph{Second order optimality conditions} \label{para:nondegen}

Throughout the paper, we will further require additional, well-established structural assumptions on \eqref{eq: problem setting}, which, on the one hand, ensure the existence of a unique, sparse minimizer $\bar{u}$ and, on the other hand, facilitate the derivation of fast convergence rates for the presented algorithms.

\begin{assumption} \label{ass:dualvariable}
    Assume that:
    \begin{enumerate}[label=\textbf{B\arabic*}]
        \item The functional $F \colon Y \to \R$ is strongly convex with a strong convexity constant $\gamma>0$ in a neighborhood $\mathcal{N}(\bar{y})$ of $\bar{y}$, i.e.
        \begin{equation*}
            (\nabla F(y_1)-\nabla F(y_2), y_1-y_2)_Y \geq \gamma \|y_1-y_2\|^2_Y \quad \text{for all} \quad y_1,y_2 \in \mathcal{N}(\bar y) .
        \end{equation*} \label{item:strongconv}
        \label{ass:strongconv}

        \item There exists a finite set $\bar{\A}=\{\bar{x}^j\}^{\bar{N}}_{j=1}$ of cardinality $\bar{N}$ with
        \begin{align*}
            \bar{\A}\subset \operatorname{int}(\Omega) , \quad  \bar{\A}=\{x\in\Omega\st\vert\bar{p}(x)\vert=\alpha=\cnorm{\bar{p}}\}.
        \end{align*}
        Moreover, the set $\{\kappa(x)\st x\in\bar{\A}\}$ is linearly independent. \label{ass:isolated}
    
        \item There is a radius $R'>0$ and a parameter $0<\sigma'<\alpha$ such that the kernel $\kappa$ satisfies
        \begin{equation*}
            \kappa\in\mathcal{C}^{2}(\overline{\Omega_{R'}};Y) , \quad \text{where} \quad \Omega_{R'}\coloneqq\bigcup^{\bar{N}}_{j=1}B_{R'}(\bar{x}^j) \subset \operatorname{int}(\Omega)
        \end{equation*}
        and there holds $\vert\bar{p}(x)\vert\leq \alpha-\sigma'$ for all $x\in\Omega \backslash \overline{\Omega_{R'}}$. Furthermore, it holds that
        \begin{equation}
        B_{2R'}(\bar{x}^j)\cap B_{2R'}(\bar{x}^i)=\emptyset 
        \end{equation}
        for all $i,j=1,\dots,\bar{N}$, $i\neq j$.
        \label{ass:regkernel}
    
        \item We have $-\sign(\bar{p}(\bar{x}^j))\nabla^2\bar{p}(\bar{x}^j) \geq_L \theta \operatorname{Id} $ for~$j=1, \dots, \bar{N}$ and some $\theta>0$. \label{ass:curvature}

        \item We have $\supp(\bar{u})= \bar{\mathcal{A}}$.\label{ass:complementary}
    \end{enumerate}
\end{assumption} 

We briefly comment on these assumptions; a more detailed account is given in \cite{PieperKonstantin2021Lcoa} as well as in \cite{wachsmuthwalter}, where the latter formulates a bridge between these assumptions and no-gap second order conditions as well as quadratic growth of $J$ w.r.t. certain unbalanced optimal transport distances.

Assumption~\eqref{ass:isolated} ensures that the solution $\bar{u}$ to \eqref{eq: problem setting} is unique and supported on the set $\bar{\mathcal{A}}$, cf. also Proposition \ref{prop:firstorder} as well as \cite[Proposition 3.8]{PieperKonstantin2021Lcoa}. As a consequence of \eqref{ass:complementary}, we have
\begin{equation*}
    \bar{u}=\sum^{\bar{N}}_{j=1} \bar{\lambda}^j \delta_{\bar{x}^j} \quad \text{for some} \quad \bar{\lambda}^j>0. 
\end{equation*}
The additional regularity provided by Assumption \eqref{ass:regkernel} further implies $K_* y \in \C \cap \C^{2}(\overline{\Omega_{R'}})$ for $y \in Y$ as well as the continuity of $K_* \colon Y \to \C \cap \C^{2} (\overline{\Omega_{R'}})$. In particular, we also have $\bar{p}\in\C \cap \C^{2}(\overline{\Omega_{R'}})$. Given that $\bar{\A} \subset \operatorname{int}(\Omega) $, we thus get $\nabla\bar{p}(\bar{x}^j)=0$.

Using Assumption~\ref{ass:dualvariable}, we can derive properties of measures $u$ contained in sublevel sets of the residual $r_J$. Given a $\Delta>0$, this set is defined as
\begin{equation}
    \E(\Delta)=\left\{u\in\Mm\st r_J(u)\leq\Delta \right\} .
\end{equation}

\begin{proposition}\label{prop: sublevel general}
    Let Assumption~\ref{ass:dualvariable} hold. Then there exists a constant $c_\Mm>0$ and a sublevel parameter $\Delta'>0$ such that for all $u\in\E(\Delta')$ the following properties hold:
    \begin{enumerate}[label=\textbf{C\arabic*}]
        \item\label{prop: sublevel y and nebla F}
        $\Vert Ku-K\bar{u}\Vert_Y\leq\sqrt{r_J(u)/\gamma}$ and $\Vert\nabla F(Ku)-\nabla F(K\bar{u})\Vert_Y\leq L\sqrt{r_J(u)/\gamma}$.

        \item\label{prop: sublevel p convergence} 
        $\cnorm{p_u-\bar{p}}\leq \cnorm{\kappa}L\sqrt{r_J(u)/\gamma}$ and $\Vert p_u-\bar{p}\Vert_{\cC^2(\overline{\Omega_{R'}})}\leq\Vert\kappa\Vert_{\cC^2(\overline{\Omega_{R'}};Y)}L\sqrt{r_J(u)/\gamma}$.

        \item\label{prop: sublevel norm convergence}
        
        $\left\vert\mnorm{u}-\mnorm{\bar{u}}\right\vert\leq c_\Mm\sqrt{r_J(u)}$.

        \item\label{prop: sublevel coef convergence}
        If $\A_{u}\subset\Omega_{R'}$, then $\mu_u^j\coloneqq \left\vert u(B_{R'}(\bar{x}^j))\right\vert\geq\frac{1}{2}\min_{j\leq \bar{N}}\vert\bar{\lambda}^j\vert\eqqcolon\bar{\mu}$ for all $j\leq \bar{N}$.
    \end{enumerate}
\end{proposition}
\begin{proof}
    See Appendix~\ref{app:optimality}.
\end{proof}

The result \eqref{prop: sublevel p convergence} can be used to derive the following properties of $p_u$:

\begin{proposition}\label{prop: sublevel p}
    Let Assumption~\ref{ass:dualvariable} hold. Then there exists a radius $0<\tilde{R}\leq R'$ such that for all radii $R\in(0,\tilde{R})$ there exist parameters $0<\Delta(R)\leq\Delta'$ and $0<\sigma(R)\leq\sigma'$ such that all $u\in\E(\Delta(R))$ satisfy:
    \begin{enumerate}[label=\textbf{D\arabic*}]        
        \item\label{prop: p local sign}
        For all $j\leq \bar{N}$, the sign of $p_u$ on $B_R(\bar{x}^j)$ is constant and satisfies
        \begin{equation}
            \sign(p_u(x))=\sign(\bar{p}(\bar{x}^j))=\sign(\bar{\lambda}^j)\quad\text{for all}\quad x\in B_R(\bar{x}^j).
        \end{equation}

        \item\label{prop: p local max} 
        For all $j\leq \bar{N}$, $\vert p_u\vert$ has a unique local maximum $\widehat{x}_u^j$ on $B_R(\bar{x}^j)$ and it holds
        \begin{equation}
            | p_u(\widehat{x}_u^j)|-| p_u(x)|\leq 2R\Vert\nabla p_u(x)\Vert\ \quad \text{for all} \quad x \in B_{R}(\bar{x}^j).
        \end{equation}

        \item\label{prop: p curvature and growth}
        The curvature and quadratic growth conditions
        \begin{equation}
            -\sign(p_u(x))\nabla^2 p_u(x) \geq (\theta/4) \operatorname{Id}
        \end{equation}
        and
        \begin{equation}
            \vert p_u(\widehat{x}_u^j)\vert-\vert p_u(x)\vert\geq\frac{\theta}{8}\Vert\widehat{x}_u^j-x\Vert^2\quad\text{for all}\quad x\in B_R(\bar{x}^j)
        \end{equation}
        hold for all $j\leq \bar{N}$.

        \item\label{prop: p sigma}
        It holds that $\vert p_u(x)\vert\leq\alpha-\sigma(R)/2$ for all $x\in\Omega \backslash \overline{\Omega_R}$.
    \end{enumerate}
\end{proposition}
\begin{proof}
    For the sake of brevity, we omit a detailed proof and point out related results, \cite[Corollary 5.11]{PieperKonstantin2021Lcoa} and \cite[Lemma 5.12]{PieperKonstantin2021Lcoa}, in the literature. 
\end{proof}

For the remainder of this work, let us fix a tuple of parameters
\begin{equation}\label{eq: parameters}
    (\gamma, \theta, R, \sigma, L, C_K, C_{K'}) \in \R^7_{++} ,
\end{equation}
where $\gamma$ and $\theta$ satisfy \eqref{item:strongconv} and \eqref{ass:curvature} respectively, the radius $R$ is as in Proposition~\ref{prop: sublevel p} with corresponding parameter and $\sigma=\sigma(R)$, and $L$, $C_K$, and $C_{K'}$ satisfy
\begin{equation}
    L_{\nabla F}\leq L ,\quad \omegaynorm{\kappa}\leq C_K ,\quad\text{and}\quad\Vert\kappa\Vert_{\cC^1(\overline{\Omega_{R'}};Y)}\leq C_{K'} .
\end{equation}

\section{A lazified generalized conditional gradient method}\label{section:lazification}

As emphasized earlier, greedy point insertion algorithms are inherently related to the \texttt{GCG} method.
Starting from a sparse initial measure~$u_1$ and given an upper bound $M>0$ on the norm of elements in~$E_J(u_1)$, the latter approximates minimizers of Problem~\eqref{eq: problem setting} by iterating 
\begin{equation}\label{eq:partiallylinearizedinoursetting}
v_k \in \argmin_{v\in\Mm,\mnorm{v}\leq M}\left\lbrack\langle -p_k,v\rangle+\alpha\mnorm{v}\right \rbrack, \quad u_{k+1}=(1-\eta_k) u_k+ \eta_k v_k ,
\end{equation}
where $v_k$ is the \texttt{GCG} direction, $p_k \coloneqq p_{u_k}$, and $\eta_k \in [0,1]$ is an appropriately chosen step size. We define the \textit{dual gap functional} associated to this problem as
    \begin{equation*}
  \Phi(u) \coloneqq \max_{v\in\Mm, \mnorm{v}\leq M}  \varphi(u,v), \quad \text{where}\quad \varphi(u,v) \coloneqq \langle p_u,v-u\rangle+\alpha\mnorm{u}-\alpha\mnorm{v},
\end{equation*}
for~$u \in E_J(u_1)$. By construction, we have 
\begin{equation*}
   v_k \in  \argmax_{v\in\Mm, \mnorm{v} \leq M}\varphi(u_k,v)\quad \Leftrightarrow\quad   v_k \in \argmin_{v\in\Mm,\mnorm{v}\leq M}\left\lbrack\langle -p_k,v\rangle+\alpha\mnorm{v}\right \rbrack.
\end{equation*}

\begin{lemma}{\cite[Proposition 5.2]{PieperKonstantin2021Lcoa}} \label{lem:resandgap}
    For every~$u \in \Mm$, we have~$\Phi(u) \geq 0$ with equality if and only if~$u$ is a minimizer of~\eqref{eq: problem setting}. Moreover, there holds~$r_J(u)  \leq \Phi(u)$.
\end{lemma}
Furthermore, we introduce a family of parametrized Dirac-Delta functions
\begin{align*}
    v_u (x)=M  \sign(p_u(x))\delta_{x}\in\Mm \quad \text{for all} \quad x \in \Omega ,
\end{align*}
as well as the following explicit characterization of a \texttt{GCG} direction~$v_k$.

\begin{lemma}{\cite[Proposition 5.3]{PieperKonstantin2021Lcoa}} \label{lem:explicitDiracsolution}
Let $u \in E_J(u_1) $ and define
\begin{align}\label{eq:solutionofPhi}
    \widehat{v}=
    \begin{cases}
        0 & \cnorm{p_{u}}<\alpha\\
        v_u (\widehat{x}) &\text{ else }
    \end{cases} , \quad \text{where} \quad \widehat{x}\in \argmax_{x\in \Omega} |p_u (x)|.
\end{align}
Then we have
\begin{align} \label{def:lineargeneral}
    \widehat{v} \in \argmax_{v\in\mathcal{M}, \mnorm{v}\leq M}\varphi(u,v) , \quad    \Phi(u)&=M\left(\cnorm{p_u}-\alpha\right)_++\alpha\mnorm{u}-\langle p_u,u\rangle.
\end{align}
\end{lemma}

Since evaluating \eqref{eq:solutionofPhi} can be expensive, we propose a lazified method. That is, instead of maximizing~$\varphi(u_k,\cdot)$, we only require that the selected direction makes it exceed a certain threshold. The specific structure in Lemma~\ref{lem:explicitDiracsolution} motivates the following definition of a lazy \texttt{GCG} direction for the problem under consideration.

\begin{definition} \label{def:lazysolution}
Given a measure~$u \in \Moc $ as well as an $\varepsilon >0$, we call
\begin{align} \label{eq:lazysolution}
    v_{\varepsilon} \in \left\{ v_u(x)\st x \in \Omega \right\} \cup \{0\} \quad \text{with} \quad    \varphi(u,v_{\varepsilon})  \geq M \varepsilon
\end{align}
a \textit{lazy direction} or \textit{lazy solution} of~\eqref{def:lineargeneral} for~$u$ at tolerance~$\varepsilon$.
\end{definition}

The resulting \texttt{GCG} method, relying on lazy solutions with adaptive tolerances~$\varepsilon=\varepsilon_k>0$ as update directions, can be found in Algorithm~\ref{alg: sublinear parameter-free}.

We make several observations. First, lazifying the insertion step allows for greater flexibility in the way of choosing~$v_{u_k}(x)$. For example, a suitable candidate point~$x$ could be found as an intermediate iterate of an optimization algorithm applied to~$|p_k|$, but also by randomly sampling points on~$\Omega$. Similarly, promising points that have been visited in earlier iterations can be cached and checked immediately for lazy optimality in the sense of Definition~\ref{def:lazysolution} in subsequent steps.

Second, in contrast to ``exact'' \texttt{GCG} directions, see Lemma~\ref{lem:explicitDiracsolution}, lazy solutions might not exist. As a consequence, the following case distinction is necessary:   
\begin{itemize}
    \item[Case 1.] If we have found a lazy solution for~$u_k$ at tolerance~$\varepsilon_k$,  we employ it as a \texttt{GCG} update direction and keep~$\varepsilon_k$ unchanged for the next iteration.  We will refer to steps of this form as ``lazy calls'' (``positive’’ calls in \cite{BraunGábor2016LCGA}). In practice, notice that we can first check whether zero is a lazy solution before considering measures of the form~$v_{u_k}(x)$ in order to further decrease the computational effort.   
    \item[Case 2.] If there is no lazy solution at the given tolerance, we perform an ``exact call'' (``negative’’ call in \cite{BraunGábor2016LCGA}). We use the update direction provided by Lemma~\ref{lem:explicitDiracsolution} in the \texttt{GCG} step. We emphasize that the computation of the latter does not entail additional effort, since the verification of the absence of lazy solutions already requires the evaluation of a supremum of $|p_k(\cdot)|$ over $\Omega$. As a by-product, we also have access to the dual gap
    \begin{align*}
        \Phi(u_k)&=M\left(\cnorm{p_k}-\alpha\right)_++\alpha\mnorm{u_k}-\langle p_k,u_k\rangle ,
    \end{align*} 
    which we use to update the tolerance~$\varepsilon_{k+1}= \Phi(u_k)/(2M)$ for the next iteration.      
\end{itemize}

This logic is presented in Algorithm~\ref{alg: lgcg step}.

Finally, we stress that the \texttt{LGCG} step~$u_{k+}$ in Algorithm~\ref{alg: sublinear parameter-free} is interpreted as an intermediate step and we only assume that the choice of the next iterate~$u_{k+1}$ satisfies~$J(u_{k+1}) \leq J(u_{k+})$. While this allows for the particular choice of~$u_{k+1}=u_{k+}$, it opens the door for the acceleration schemes introduced in the following sections.

The remainder of this section is dedicated to the proof of a sublinear rate of convergence for Algorithm~\ref{alg: sublinear parameter-free}. 

\begin{algorithm}
\caption{\textup{\texttt{LGCGStep}}}\label{alg: lgcg step}
\KwIn{Measure $u$, threshold $\varepsilon$, constant $C$}
\KwOut{Updated measure $u_+$, update direction $v$, updated threshold $\e_+$}
Find a lazy solution~$v_{\varepsilon}$ of~\eqref{def:lineargeneral} for~$u$ at tolerance~$\varepsilon$\\
\eIf{\textup{lazy call}}{$\eta\leftarrow \min\left\{1,\frac{M\varepsilon}{C}\right\}, v\leftarrow v_{\varepsilon}, \varepsilon_{+}\leftarrow\varepsilon$}
{
Update
\begin{equation*}
v\leftarrow
\begin{cases}
    0 & \cnorm{p_u}<\alpha\\
     v_{u} (\widehat{x}) &\text{ else } 
\end{cases} , \quad \text{where} \quad \widehat{x}\in \argmax_{x\in \Omega} |p_u(x)|   
\end{equation*}\\
$\Phi(u)\leftarrow \varphi(u,v)$\\
$\eta\leftarrow \min\left\{1,\frac{\Phi(u)}{C}\right\}, \varepsilon_{+}\leftarrow \Phi(u)/(2M)$
}
$u_{+}\leftarrow(1-\eta)u+ \eta v$\\
\Return{$u_{+}, v, \varepsilon_{+}$}
\end{algorithm}

\begin{algorithm}
\caption{Lazified Generalized Conditional Gradient (\textup{\texttt{LGCG}})}
\label{alg: sublinear parameter-free}
\KwIn{Initial iterate $u_1$, initial threshold $\varepsilon_1$, constant $C=4LM^2C^2_K$}
\For{$k=1,2,\cdots$}{
$u_{k+}, v_k, \varepsilon_{k+1}\leftarrow\textup{\texttt{LGCGStep}}(u_k,\varepsilon_k,C)$\\
\If{$\varepsilon_{k+1}=0$}{Terminate with~$u_k$ a minimizer of \eqref{eq: problem setting}}
Find $u_{k+1}\in\Mm$ with $J(u_{k+1})\leq J(u_{k+})$\label{line: improve}
}
\end{algorithm}

We begin by showing a few useful properties of Algorithm~\ref{alg: lgcg step}.

\begin{lemma}\label{lemma: LGCG in Newton setting}
    Consider some measure $u\in\Mm$, a corresponding dual variable $p_u$, some threshold $\e$, and the constant $C=4LM^2C^2_K$. Let $(u_+, v, \e_+)$ be the output of~\textup{\texttt{LGCGStep}}$(u,\e,p_u,C)$. Then it holds that
    \begin{equation}\label{eq: LGCG descent}
        J(u_+)-J(u)\leq
        \begin{cases}
            -\frac{M^2\e_+^2}{2C} &,\quad M\e_+\leq C\\
            \frac{C}{2}-M\e_+ &,\quad\text{ else }
        \end{cases} .
    \end{equation}
    In particular, we have $J(u_+)\leq J(u)$. Furthermore, if $\e$ is such that $r_J(u)\leq 2M\e$, then it holds that
    \begin{equation}
        r_J(u_+)\leq r_J(u)\leq 2M\e_+\leq2M\e.
    \end{equation}
\end{lemma}
\begin{proof}
    Using Taylor expansion we obtain
    \begin{equation}
        J(u_+)-J(u)\leq \eta \left( \dual{p_u}{v-u}+ \alpha \mnorm{u}- \alpha \mnorm{v} \right)+ C \frac{\eta^2}{2} \notag= -\eta \varphi(u,v)+ C \frac{\eta^2}{2} .
    \end{equation}
    First, consider the case of a lazy call. In this case, $v=v_\e$ is a lazy solution, which yields
    \begin{equation}
        J(u_+)-J(u)\leq -\eta M\varepsilon+ C \frac{\eta^2}{2} .
    \end{equation}
    Notice that the step size, given by $\eta=\min\left\{1,\frac{M\varepsilon}{C}\right\}$, is in fact a minimizer over $[0,1]$ of the quadratic equation on the right-hand side of the above inequality. With direct computation, we obtain
    \begin{equation}\label{eq: quadratic cases}
        J(u_+)-J(u)\leq\min_{\eta\in[0,1]}\left[-\eta M\varepsilon+ C \frac{\eta^2}{2}\right]=
        \begin{cases}
            -\frac{M^2\e^2}{2C} &,\quad M\e\leq C\\
            \frac{C}{2}-M\e &,\quad\text{ else }
        \end{cases} .
    \end{equation}
    Noticing that $\e=\e_+$ after a lazy call concludes the proof of this case.

    In the case of an exact call, $v$ is such that $\varphi(u,v)=\Phi(u)$. We obtain
    \begin{equation}
        J(u_{+})-J(u) \leq -\eta \Phi(u) + C\frac{\eta^2}{2} .
    \end{equation}
    Once again, the choice of $\eta$ minimizes the quadratic equation on the right-hand side, so we can write, substituting the definition of $\e_+$,
    \begin{equation}
        J(u_+)-J(u)\leq\min_{\eta \in [0,1]} \left \lbrack -\eta\Phi(u)+ C \frac{\eta^2}{2} \right \rbrack=\min_{\eta \in [0,1]} \left \lbrack -2\eta M\e_++ C \frac{\eta^2}{2} \right \rbrack\leq\min_{\eta \in [0,1]} \left \lbrack -\eta M\e_++ C \frac{\eta^2}{2} \right \rbrack .
    \end{equation}
    The same computation as in \eqref{eq: quadratic cases} concludes the proof of the first statement.

    As for the second statement, notice that the above implies $r_J(u_+)\leq r_J(u)$. In the case of a lazy call, the inequality follows from $\e=\e_+$. In the case of an exact call, we can write, using Lemma~\ref{lem:resandgap},
    \begin{equation}
        r_J(u)\leq\Phi(u)=2M\e_+\leq M\e,
    \end{equation}
    where the last inequality follows from the property $\Phi(u)\leq M\e$ implied by the inexistence of a lazy solution.
\end{proof}

\begin{theorem}\label{theorem: sublinear}
    Let~$\epsilon>0$ be arbitrary but fixed. Assume that Algorithm~\ref{alg: sublinear parameter-free} generates an infinite sequence $(u_k)$ of iterates. If the initial tolerance $\varepsilon_1$ satisfies $r_J(u_1)\leq2M\varepsilon_1$, then we have
   \begin{align*}
       J(u_{k+1})\leq J(u_k) \quad \text{as well as} \quad r_J(u_k) \leq 2 M \varepsilon_k  
   \end{align*} 
    for all~$k\in\N$. Moreover, there holds $r_J (u_k)\leq\epsilon$ for all~$k \geq \bar{k}(\epsilon)$, where~$ \bar{k}(\epsilon)$ satisfies
    \begin{equation*}
        \bar{k}(\epsilon)\leq \left\lceil\log_2\frac{M\varepsilon_1}{\epsilon}\right\rceil+1+4\left\lceil\log_2\frac{M\varepsilon_1}{C}\right\rceil+64 \frac{C}{\epsilon}.
    \end{equation*}
    In particular, we have $r_J(u_k)= \mathcal{O}(1/k)$ and, if Assumption~\ref{ass:dualvariable} holds, also $u_k\rightharpoonup^*\bar{u}$.
\end{theorem}
\begin{proof}
The first statement follows inductively from Lemma~\ref{lemma: LGCG in Newton setting} and $J(u_{k+1})\leq J(u_{k+})$.

Let us now prove the complexity estimate. The threshold $\e_k$ only changes during an exact call. In such a case, it holds that $\Phi(u_k) < M \varepsilon_k$  and thus, by definition,~$\varepsilon_{k+1} < \varepsilon_k /2$. In particular, $\seq{\varepsilon_k}$ is a decreasing sequence. Furthermore, $\varepsilon_k >0$ for all~$k \geq 1$, since Algorithm~\ref{alg: sublinear parameter-free} does not converge in finitely many steps by assumption. Using $r_J(u_{k})\leq 2M \varepsilon_{k}$, we conclude that at most $\lceil\log_2\frac{M\varepsilon_1}{\epsilon}\rceil+1$ exact calls are encountered until we have~$r_J(u_k)\leq \epsilon$.
 
It remains to count the number~$k'$ of lazy calls following an exact one. The initial number of lazy calls at the start of the iteration can be bounded analogously. For this purpose, let~$k,k' \in \N$ be such that iteration~$k$ corresponds to an exact call and the following~$k'$ iterations are lazy calls. Then we have~$\varepsilon_{k+1}=\varepsilon_{k+1}=\cdots=\varepsilon_{k+k'}$. With Lemma~\ref{lemma: LGCG in Newton setting}, we can write
\begin{align} \label{eq:lowerestcasedist}
     2M \varepsilon_{k+1} \geq  r_J(u_{k+1}) \geq J(u_{k+1})-J(u_{k+k'+1}) &= \sum_{i=k+1}^{k+k'}(J(u_i)-J(u_{i+1})) \notag\\ &\geq 
            \begin{cases}
                k'\frac{M^2\varepsilon_{k+1}^2}{2C}&,\quad M\varepsilon_{k+1}\leq C\\
                k'(M\varepsilon_{k+1}-\frac{C}{2})&,\quad\text{else}
            \end{cases} ,
\end{align}
We make a case distinction:
\begin{itemize}
    \item[Case 1:] If~$M\varepsilon_{k+1} > C$, we use~\eqref{eq:lowerestcasedist} to conclude
       \begin{align*}
              k' \leq \frac{2M\varepsilon_{k+1}}{M\varepsilon_{k+1}-\frac{C}{2}}=\frac{4M\varepsilon_{k+1}}{2M\varepsilon_{k+1}-C}\leq\frac{4M\varepsilon_{k+1}}{2M\varepsilon_{k+1}-M\varepsilon_{k+1}}=4.
       \end{align*}
       Moreover, since the update rule for exact calls at least halves the tolerance, this case can only happen at most~$\lceil\log_2 ((M\varepsilon_1)/C)\rceil$ times, yielding in the worst-case $4\lceil\log_2 ((M\varepsilon_1)/C)\rceil$ iterations.
    \item[Case 2:] If~$M\varepsilon_{k+1} \leq C$, we recall that
       \begin{align*}
           r_J(u_{k})\leq2M \varepsilon_{k+1} \leq 2C 
       \end{align*}
       Since we are interested in the worst-case behavior, we can further assume that~$2M \varepsilon_{k+1} > \epsilon$. The latter  implies that there is an $\ell_k \in \N$ with
       \begin{align*}
       2^{-\ell_k-1} C \leq M\varepsilon_{k+1} \leq 2^{-\ell_k} C  \quad \text{as well as} \quad \ell_k \leq \lceil \log_2 (C/\epsilon) \rceil +1.
       \end{align*}
       Thus,~\eqref{eq:lowerestcasedist} implies that $k' \leq 2^{\ell_k +3}$.
       Moreover, if~$k_1, k_2 \in \N$ are two indices corresponding to exact calls with~$M\varepsilon_{k_1+1} \leq C$ and $M\varepsilon_{k_2+1} \leq C$, respectively, as well as~$k_1 < k_2$, we conclude~$\ell_{k_1} +1 \leq \ell_{k_2}$ since consecutive exact calls at least halve the tolerance. As a consequence, the combined number of iterations in this case is bounded by
       \begin{align*}
           \sum^{\lceil \log_2 (C/\epsilon) \rceil +1}_{j=0} 2^{j +3} \leq 2^{\lceil \log_2 (C/\epsilon) \rceil +5} \leq 2^{ \log_2 (C/\epsilon)   +6}= 64 \frac{C}{\epsilon}. 
       \end{align*} 
\end{itemize}
Combining both cases with the number of potential exact calls yields the desired statement.

In order to see the sublinear rate of convergence, we set
\begin{align*}
    c_1=M \varepsilon_1+ 64C , \quad c_2 =2+4\left\lceil\log_2\frac{M\varepsilon_1}{C}\right\rceil ,
\end{align*}
and let~$k \geq c_2 +1$ be arbitrary but fixed. Setting~$\epsilon(k)= c_1/(k-c_2)$, we note that
\begin{equation}
    \bar{k}(\epsilon(k))\leq \frac{M\e_1+64C}{\epsilon(k)}+2+4\left\lceil\log_2\frac{M\varepsilon_1}{C}\right\rceil=\frac{c_1}{\epsilon(k)}+c_2=k .
\end{equation}
Thus, by the definition of $\bar{k}(\epsilon)$,
\begin{align*}
    r_J(u_k)\leq r_J(u_{\bar{k}(\epsilon(k))})\leq \epsilon(k)= \frac{c_1}{k-c_2}.  
\end{align*}

Finally, the weak* convergence follows from $r_J(u_k)\rightarrow0$ like in the proof of Proposition~\ref{prop: sublevel general}.
\end{proof}

\section{Lazifying Primal-Dual Active Point methods}\label{section:lpdap}

Following the program established in the previous section, our interest now lies in relaxing the Primal-Dual-Active Point method (\texttt{PDAP}), proposed in  \cite{PieperKonstantin2021Lcoa}, which can be stated as
\begin{equation} \label{def:pdapactualsection}
    u_{k+1} \in \argmin_{u \in \mathcal{M}(\A_k \cup\{\widehat{x}_k\}) } J(u) \quad \text{with} \quad \A_k= \A_{u_k} \quad \text{and} \quad \widehat{x}_k \in \argmax_{x \in \Omega} |p_k(x)| ,
\end{equation}
where we replace the spatial domain $\Omega$ by a finite set of distinct points $\A_k \cup\{\widehat{x}_k\}$ in the update of the iterate. While \texttt{PDAP} retains the worst-case convergence guarantees of \texttt{GCG}, it also ensures that both the support of the iterate $u_k$ as well as the new candidate point $\widehat{x}_k$ cluster around the support of $\bar{u}$ provided that Assumption \ref{ass:dualvariable} holds. In the following, we show that these favorable properties are retained, and can be exploited, for a lazified version of \eqref{def:pdapactualsection}, eventually leading to an asymptotic linear rate of convergence.

For this purpose, and for a finite set of distinct points $\mathcal{A}=\{x^j\}^N_{j=1}$, consider the coefficient update problem
\begin{equation}\tag{$\mathcal{P}_\A$}\label{eq: low-dimensional problem}
    \min_{u\in\mathcal{M}(\A)}J(u),
\end{equation}
noting that
\begin{equation*}
\M{\A}= \left\{ u_\lambda\;|\;u_\lambda=\sum^N_{j=1}\lambda^j \delta_{x^j},~\lambda \in \R^N \right\}, \quad J(u_\lambda)= F \left(\sum^N_{j=1} \lambda^j \kappa(x^j) \right)+ \alpha \Vert\lambda\Vert_{\ell_1}.
\end{equation*}
As a consequence, \eqref{eq: low-dimensional problem} corresponds to to a finite-dimensional, convex but nonsmooth minimization problem for which we have
\begin{align*}
    J(u)- \min_{v \in \mathcal{M}(\A)} J(v) \leq \Phi_{\A}(u), \quad \Phi_{\A}(u)   \coloneqq \max_{v\in\mathcal{M}(\A), \mnorm{v} \leq M}\varphi(u,v) ,
\end{align*}
as well as
\begin{align} \label{eq:combinedgap}
\Phi_{\A}(u)&=M\left(\max_{x\in\A}\vert p_u(x)\vert-\alpha\right)_++\alpha\mnorm{u}-\langle p_u,u\rangle, \\  \Phi(u)&= M\left(\cnorm{p_u}-\max\left\{\max_{x\in\A}\vert p_u(x)\vert , \alpha \right\}\right)_++\Phi_\A(u) 
\end{align}
in view of Lemmas \ref{lem:resandgap} and \ref{lem:explicitDiracsolution}, respectively. Note that, in contrast to $\Phi(u)$, $\Phi_{\A}(u)$ is exactly computable in $\# \A$ operations. 

In order to increase readability, we focus on the main results in the following exposition and move the proofs of necessary auxiliary results to Appendix \ref{app:lpdap}.

Loosely speaking, we lazify \eqref{def:pdapactualsection} by replacing the exact computation of $\widehat{x}_k$ by a lazy update step in the spirit of Section \ref{section:lazification} and allowing for an inexact resolution of the coefficient update problem \eqref{eq: low-dimensional problem}, which is controlled by the gap functional $\Phi_{\mathcal{A}_k}(u_k)$, where $\A_k=\A_{u_k}$. The former is further augmented by the \textit{Local Support Improver} (\texttt{LSI}), which exploits the local strong concavity of $p_k$, while the latter is facilitated by \textit{Drop Steps}, removing Dirac-Delta functionals far away from $\bar{\A}$. Furthermore, the coefficient update problem \eqref{eq: low-dimensional problem} is modified to asymptotically optimize only over measures with the desired sign on $\Omega_R$.

For the remainder of this work, let Assumptions~\ref{ass:functions} and \ref{ass:dualvariable} hold and use the parameters defined in \eqref{eq: parameters}. Let $M$ be as in Section~\ref{section:lazification}.


We start by describing the drop step. For this, consider the set
\begin{align}\label{eq: drop set}
    \mathcal{D}_u \coloneqq \left\{ x \in \A_u \;|\; \sign(p_u(x)) \neq \sign(u(\{x\})) \lor |p_u(x)| \leq \alpha- \sigma/2   \right\}
\end{align}
and define the drop measure associated to $u$ as $ u^{\text{drop}} \coloneqq u \mres (\Omega\setminus \mathcal{D}_u)$.

\begin{lemma}\label{lemma: drop valid}
Let $\Delta(R)$ be as in Proposition~\ref{prop: sublevel p}. Then there exists a $0<\Delta\leq\Delta(R)$ such that for all sparse $u\in\Mm$ with $r_J(u)\leq\Delta$ it holds $J(u^{\textup{drop}}) \leq J(u)$, as well as 
\begin{align}\label{eq: drop properties}
    \A_{u^\textup{drop}}\cap B_R(\bar{x}^j)\neq\emptyset\quad\text{for all}\quad j\leq\bar{N},\quad \A_{u^\textup{drop}}\subset\Omega_R,\quad\text{and}\\\sign(p_{u^\textup{drop}}(x))=\sign(u^\textup{drop}(\{x\}))\quad\text{for all}\quad x\in\A_{u^{\textup{drop}}}.
\end{align}
\end{lemma}
\begin{proof}
    See Appendix~\ref{app:lpdap}.
\end{proof}

This motivates Algorithm~\ref{alg: drop step}.

\begin{algorithm} 
    \caption{\textup{\texttt{DropStep}}}
    \label{alg: drop step}
    \KwIn{Measure $u$}
    \KwOut{Improved measure $u_+$}
    $\D_u\leftarrow\left\{ x \in \A_u \;|\; \sign(p_u(x)) \neq \sign(u(\{x\})) \lor |p_u(x)| \leq \alpha- \sigma/2   \right\}$\\
    $u^\text{drop}\leftarrow u\mres (\Omega\setminus \mathcal{D}_u)$\\
    \eIf{$J(u^\textup{drop})\leq J(u)$}{
        $u_+\leftarrow u^\text{drop}$
    }
    {$u_+\leftarrow u$}
    \Return{$u_+$}
\end{algorithm}

Thus, for sparse measures $u$ with small enough objective functional value, Lemma~\ref{lemma: drop valid} ensures that the output $u_+$ of \texttt{DropStep}$(u)$ satisfies \eqref{eq: drop properties}.

Next, we carefully relax the exact resolution of the coefficient update problem \eqref{eq: low-dimensional problem} with a particular focus on guaranteeing the compatibility condition on the sign from \eqref{eq: drop properties}. 

For some nonzero sparse measure $u\in\Mm$, let its support be given by $\A_u=\{x^j\}_{j=1}^N$, $N\in\N$. Consider the modified problem
\begin{equation}\tag{$\mathcal{P}^u$}\label{eq: positive coefficient update}
    \min_{w\in\Mm_+(\A_u)}J^u(w) ,\quad\text{where}\quad J^u(w)=F\left(K^uw\right)+\alpha\mnorm{w}
\end{equation}
and the linear operator $K^u:\Mm(\A_u)\rightarrow Y$ is given by
\begin{equation}
    K^uw=\sum_{j=1}^N\kappa(x^j)\sign(u(\{x^j\}))w(\{x^j\})\,.
\end{equation}
Notice that we minimize over the space of positive measures $\Mm_+(\A_u)$ with support contained in $\A_u$ while the effective sign of each Dirac Delta is fixed by definition of $K^u$.

More in detail, for $w\in\Mm_+(\A_u)$, it holds
\begin{equation}\label{eq: definition of v^u_i}
    J^u(w)=J(v^u_w) ,\quad\text{where}\quad v^u_w=\sum_{j=1}^N\sign(u(\{x^j\}))w(\{x^j\})\delta_{x^j}\in\Mm(\A_u).
\end{equation}
The associated dual variable $p_w^u\in\cC(\A_u)$ is given by
\begin{equation}
    p_{w}^u(x^j)=-K^u_*\nabla F\left(K^uw\right)(x^j)
\end{equation}
for $j\leq N$. Similarly, we obtain the primal-dual gap
\begin{equation}
    \Phi^u(w)=\max_{v\in\Mm_+(\A_u),\mnorm{v}\leq M}\varphi^u(w,v) ,\quad\text{where}\quad\varphi^u(w,v)=\langle p^u_w , v-w\rangle+\alpha\mnorm{w}-\alpha\mnorm{v}.
\end{equation}
which we can rewrite as
\begin{equation}\label{eq: explicit Phi u}
    \Phi^u(w)=M\left(\max_{j\leq N} p_{w}^u(x^j)-\alpha\right)_++\alpha\mnorm{w}-\langle p_w^u ,w\rangle
\end{equation}
for all $w\in\Mm_+(\A_u)$.

\begin{lemma}\label{lemma: coefficient same Phi}
    For all sparse $u\in\Mm$ with $\A_u\subset\Omega_R$, $\sign(u(\{x\}))=\sign(p_u(x))$ for all $x\in\A_u$, and $r_J(u)$ small enough it holds that $\Phi^u(w)=\Phi_{\A_u}(v^u_w)$ for all $w\in\Mm_+(\A_u)$ with $J(v^u_w)\leq J(u)$.
\end{lemma}
\begin{proof}
    See Appendix~\ref{app:lpdap}
\end{proof}

\begin{algorithm} 
    \caption{\textup{\texttt{CoefficientStep}}}
    \label{alg: low-dimensional problem}
    \KwIn{Measure $u$, accuracy $\Psi>0$}
    \KwOut{Improved measure $u_+$, positive measure $w_+$}
    $w_0\leftarrow\sum_{x\in\A_u}\bigl \vert u(\{x\})\vert\delta_x\in\Mm_+(\A_u)$\\
    Find a $w_+\in\Mm_+(\A_u)$ such that $J^u(w_+)\leq J^u(w_0)$ and $\Phi^u(w_+)\leq\Psi$\label{line: cefficient problem in algo}\\
    $u_+\leftarrow v_{w_+}^u$\\
    \Return{$u_+,w_+$}
\end{algorithm}

Consider Algorithm~\ref{alg: low-dimensional problem}. For all sparse $u$ that satisfy the conditions of Lemma~\ref{lemma: coefficient same Phi}, this algorithm returns measures that solve both \eqref{eq: low-dimensional problem} and \eqref{eq: positive coefficient update} up to the given accuracy $\Psi$.

\begin{lemma}\label{lemma: output of coefficient step}
    For all sparse $u\in\Mm$ with $\A_u\subset\Omega_R$, $\sign(u(\{x\}))=\sign(p_u(x))$ for all $x\in\A_u$, and $r_J(u)$ small enough it holds that the output $u_+,w_+$ of \textup{\texttt{CoefficientStep}}$(u,\Psi)$ satisfies $\sign(u_+(\{x\}))=\sign(p_{u_+}(x))$ for all $x\in\A_{u_+}$ as well as $\Phi_{\A_{u_+}}(u_+)\leq\Phi^{u}(w_+)\leq\Psi$.
\end{lemma}
\begin{proof}
    See Appendix~\ref{app:lpdap}.
\end{proof}

In the following, we want to exploit the structure of $p_u$ given by Propositions~\ref{prop: sublevel general} and \ref{prop: sublevel p} to locally improve support points $x\in\A_u$ in a way that allows for the construction of refined descent directions and facilitates the computation of lazy solutions to \eqref{eq:partiallylinearizedinoursetting}. More in detail, given an $x\in\A_u$, we look for a point $x_{\text{LSI}} \in B_{2R}(x)$ with
\begin{equation}
    \label{eq: LSI sigma}
    \vert p_u(x_{LSI})\vert>\alpha-\sigma/2 ,
\end{equation}
\begin{equation}\label{eq: LSI suboptimality ineuqality}    
    \vert p_u(x_{\text{LSI}})\vert-\max_{z\in \A_u\cap B_{2R}(x)}\vert p_u(z)\vert\geq 2R\Vert\nabla p_u(x_{\text{LSI}})\Vert ,
\end{equation}
and
\begin{align}\label{eq: LSI Phi inequality}
    \Vert\nabla &p_u(x_{\text{LSI}})\Vert\leq\Phi_{\A_u}(u) ,
\end{align}
which reflect our desire to compute local maximizers of $|p_u|$ as potential candidates for lazy update directions. In this context, the enlarged balls $B_{2R}(x)$ serve as a proxy for the unknown neighborhoods $B_{R}(\bar{x}^j)$, noting that $\A \subset \Omega_R$ implies
\begin{equation*}
    \A\cap B_{2R}(x)= \A \cap B_{R}(\bar{x}^j) \quad \text{for all} \quad x \in B_{R}(\bar{x}^j)
\end{equation*}
by \eqref{ass:regkernel} if $r_J(u)$ is small enough. The resulting subroutine, called the Local Support Improver (\texttt{LSI}), is summarized in Algorithm \ref{alg: local improver}. Note that the described procedure is not applied to every $x \in \A_u$, but instead we successively construct a covering of $\A_u$ by balls of radius $2R$, owing to the fact that support points of $u$ can cluster.

\begin{algorithm}
\caption{Local Support Improver (\texttt{LSI})}
\label{alg: local improver}
\KwIn{Measure $u$}
\KwOut{Sets of improved points $\mathcal{B}$}
$\A\leftarrow\A_u$\\
$\mathcal{B}\leftarrow\emptyset$\\
\While{$\A\neq\emptyset$}{
  Choose $x \in \argmax_{z\in\A}\vert p_u(z)\vert$\label{line: start point selection}\\
    Find, if one exists, an $x_{\text{LSI}}\in B_{2R}(x)$ satisfying
    \begin{align*}
     \vert p_u(x_{LSI})\vert>\alpha-\sigma/2 ,\quad\Vert\nabla p_u(x_{\text{LSI}})\Vert\leq\Phi_{\A_u}(u) ,
    \end{align*}
    as well as
    \begin{align*}
        |p_u(x_{\text{LSI}})|-\max_{z\in \A_u\cap B_{2R}(x)}\vert p_u(z)\vert\geq 2R\Vert\nabla p_u(x_{\text{LSI}})\Vert .
    \end{align*}\\
    $\mathcal{B}\leftarrow \mathcal{B} \cup \{x_{\text{LSI}}\}$\\
    $\A\leftarrow\A \backslash  B_{2R}(x)$
}
\Return{$\mathcal{B}$}
\end{algorithm}

The following lemma shows that Algorithm~\ref{alg: local improver} is well defined.

\begin{lemma}\label{lemma: LSI valid}
    If the radius $R$ is as in \eqref{eq: parameters}, then for all $u\in\Mm$ with $\A_u\subset\Omega_R$ and $r_J(u)$ small enough Algorithm \ref{alg: local improver} produces a set $\mathcal{B}_u=\{x^{u,j}_{\text{LSI}}\}^{\bar{N}}_{j=1}$ with $x^{u,j}_{\text{LSI}} \in B_R(\bar{x}^j)$ for all $j\leq \bar{N}$.

    In particular, it also holds that
    \begin{equation*}
        \A_u\cap B_{2R}(x^{u,j}_{\text{LSI}})=  \A_u\cap B_R(\bar{x}^j)
    \end{equation*}
    for all $j\leq \bar{N}$.
\end{lemma}
\begin{proof}
    See Appendix~\ref{app:lpdap}.
\end{proof}
 
Let $\mathcal{B}_u$ be the output of \texttt{LSI}$(u)$. By construction, elements in $\mathcal{B}_u$ allow for a tight estimation of the suboptimality of points in $\A_u$. For this purpose, recall that, for all $u$ such that $r_J(u)$ is small enough, there is an index $\bar{\jmath}_u \in \{1,\dots,\bar{N}\}$ such that $\widehat{x}_u \coloneqq \widehat{x}_u^{\bar{\jmath}_k}$ is a global maximizer of $|p_u|$, see Propositions~\ref{prop: sublevel p} and \ref{prop: sublevel general}. 

\begin{lemma}\label{lemma: LSI suboptimal}
   For all $u\in \Mm$ with $\A_u\subset\Omega_R$ and $r_J(u)$ small enough it holds 
    \begin{equation}
        \label{eq:maximum bound}
        \vert p_u(\widehat{x}^u_{\text{LSI}})\vert-\max_{x\in \A_u\cap B_{2R}(\widehat{x}^u_{\text{LSI}})}\vert p_u(x)\vert\geq\frac{1}{2}\left(\vert p_u(\widehat{x}_u)\vert-\max_{x\in\A_u\cap B_R(\bar{x}^{\bar{\jmath}u})}\vert p_u(x)\vert \right),
    \end{equation}
    where
    \begin{equation*}
        \widehat{x}^u_{\text{LSI}}\in\argmax_{x\in\mathcal{B}_u}\left\lbrack\vert p_u(x)\vert-\max_{z\in \A_u\cap B_{2R}(x)}\vert p_u(z)\vert\right\rbrack.
    \end{equation*}
\end{lemma}
\begin{proof}
    See Appendix~\ref{app:lpdap}.
\end{proof}

Once $\mathcal{B}_u$ is computed, we use the improved support points to construct a new update direction $\tilde{v}_u$ by lumping the mass of $u$ around elements of $\mathcal{B}_u$. In view of Lemma \ref{lemma: LSI valid}, we have
\begin{align*}
    \tilde{v}_u=\sum^{\#\B_u}_{j=1} u(B_{2R}(x^{u,j}_{\text{LSI}})) \delta_{x^{u,j}_{\text{LSI}}}= \sum^{\bar{N}}_{j=1} u(B_R(\bar{x}^j)) \delta_{x^{u,j}_{\text{LSI}}}
\end{align*}
for all $u\in \Mm$ with $\A_u\subset\Omega_R$ and $r_J(u)$ small enough. The following results show that using $\tilde{v}_u$ as an alternative to the lazy update direction leads to a linear decrease of the residual, provided that the local dual gap $\Phi_{\A_u}(u)$ and residual $r_J(u)$ are small enough.

\begin{lemma}\label{lemma:bound on K}
For all $u\in\Mm$ with $\A_u\subset\Omega_R$ and $r_J(u)$ small enough we have  
    \begin{equation}\label{eq: K inequality}
        \Vert K(\tilde{v}_u-u)\Vert_Y\leq \tilde{C}\sqrt{\Phi(u)},
    \end{equation}
    where
    \begin{equation} \label{eq:cconstant}
        \tilde{C}=2C_{K'}\left(2M \sqrt{\frac{R}{\theta}}+\frac{2M C_{K'}L}{\theta\sqrt{\gamma}}+\sqrt{\frac{M}{\theta}}\right).
    \end{equation}
\end{lemma}
\begin{proof}
    See Appendix~\ref{app:lpdap}.
\end{proof}

\begin{theorem}\label{thm:local linear}
There exists a $\zeta \in (0,1)$ such that for all sparse $u\in\Mm$ with $\A_u\subset\Omega_R$, $\sign(u(\{x\}))=\sign(p_u(x))$ for all $x\in\A_u$, and $r_J(u)$ small enough there is a $\tilde{\eta}_u\in [0,1]$ such that $\tilde{u}_{+} \coloneqq u+ \tilde{\eta}_u(\tilde{v}_u-u)$ satisfies
    \begin{equation*}
        r_J(\tilde{u}_{+})\leq\zeta r_J(u)
    \end{equation*}
whenever $\Phi_{\A_u}(u) \leq \Phi(u)/2$.
\end{theorem}
\begin{proof}
    Let $\eta \in [0,1]$ be arbitrary but fixed. A Taylor expansion reveals 
    \begin{equation}\label{eq: first linear decomposition}
        r_J(u+ \eta (\tilde{v}_u-u))\leq r_J(u)+\eta\langle p_u , u-\tilde{v}_u\rangle+\frac{L}{2}\eta^2\Vert K(\tilde{v}_u-u)\Vert_Y^2.
    \end{equation}
    Since $u$ is not optimal, we have $0 \leq \Phi_{\A_u}(u)<\Phi(u)$ by assumption and thus
    \begin{equation}\label{eq:simple Phi}
        \Phi(u)=M\left(\cnorm{p_u}-\max\left\{\max_{x\in\A_u}\vert p_u(x)\vert, \alpha\right\}\right)+\Phi_{\A_u}(u)
    \end{equation}
according to \eqref{eq:combinedgap}. From Lemma \ref{lemma:bound on K}, we get
    \begin{align*}
        \frac{L}{2}\Vert&K(\tilde{v}_u-u)\Vert_Y^2\leq\frac{L \tilde{C}^2}{2} \Phi(u).
    \end{align*} 
Further recall that, for $j=1, \dots, \bar{N}$, $\A_u \cap B_{2R}(x^{u,j}_{\text{LSI}})=  \A_u \cap B_R(\bar{x}^j)$, $x^{u,j}_{\text{LSI}}\in B_R(\bar{x}^j)$, as well as that $p_u$ doesn't change sign on $B_R(\bar{x}^j)$, see Lemma~\ref{lemma: LSI valid} and Proposition~\ref{prop: sublevel p}. As a consequence, and due to $\A_u\subset\Omega_R$ and $\sign(u(\{x\}))=\sign(p_u(x))$ for all $x\in\A_u$, we have
\begin{align*}
    \langle p_u, u-\tilde{v}_u\rangle &=\sum^{\bar{N}}_{j=1} \sum_{x \in \A_u \cap B_R(\bar{x}^j)  } |u(x)| \left( |p_u(x)|-|p_u(x^{u,j}_{\text{LSI}})| \right)
    \\ & \leq \sum^{\bar{N}}_{j=1} \mu_u^j \max_{x \in \A_u \cap B_R(\bar{x}^j)} \left( |p_u(x)|-|p_u(x^{u,j}_{\text{LSI}})| \right) ,
\end{align*}
where $\mu_u^j$ is as in \eqref{prop: sublevel coef convergence}. Defining $\widehat{x}^u_{\text{LSI}}$ as in Lemma \ref{lemma: LSI suboptimal}, setting $\widehat{\mu}_u \coloneqq |u(B_{2R}(\widehat{x}^{u}_{\text{LSI}}))| $, as well as noting that the terms in the brackets are nonpositive by construction, we finally conclude
    \begin{align*}
        \langle p_u, u-\tilde{v}_u\rangle& \leq \widehat{\mu}_u\left(\max_{x\in \A_u\cap B_{2R}(\widehat{x}^{u}_{\text{LSI}})}\vert p_u(x)\vert-\vert p_u(\widehat{x}^{u}_{\text{LSI}})\vert\right)\leq\frac{\widehat{\mu}_u}{2}\left(\max_{x\in \A_u}\vert p_u(x)\vert-\cnorm{p_u}\right)\\
        &\leq\frac{\widehat{\mu}_u}{2}\left(\max\left\{\alpha , \max_{x\in\A_u}\vert p_u(x)\vert\right\}-\cnorm{p_u}\right)\\
        &=\frac{\widehat{\mu}_u}{2}\left(\max\left\{\alpha , \max_{x\in\A_u}\vert p_u(x)\vert\right\}-\cnorm{p_u}-\frac{1}{M}\Phi_{\A_u}(u)\right)+\frac{\widehat{\mu}_u}{2M}\Phi_{\A_u}(u)\\
        &=-\frac{\widehat{\mu}_u}{2M}\Phi(u)+\frac{\widehat{\mu}_u}{2M}\Phi_{\A_u}(u) \leq -\frac{\widehat{\mu}_u}{4M}\Phi(u),
    \end{align*}
where the second inequality follows from Lemma \ref{lemma: LSI suboptimal}, the equality on the third line holds due to \eqref{eq:simple Phi}, and the final inequality is due to $\Phi_{\A_u}(u) \leq \Phi(u)/2$.
In summary, we obtain
    \begin{equation}\label{eq: together bound}
        r_J\left(u+ \eta (\tilde{v}_u-u)\right)\leq r_J(u)+\left(-\frac{\widehat{\mu}_u}{4M}\eta+\frac{L\tilde{C}^2}{2}\eta^2\right)\Phi(u) \quad \text{for all} \quad \eta \in [0,1].
    \end{equation}
    where the right-hand side is minimized by 
    \begin{equation}\label{eq: linear step size}
       \tilde{\eta}_u \coloneqq \min\left\{1, \frac{\widehat \mu_u}{4ML\tilde{C}^2}\right\}.
    \end{equation}  
    Setting this value in \eqref{eq: together bound},
    \begin{align}
        -\frac{\widehat{\mu}_u}{4M}\tilde{\eta}_u+\frac{L\tilde{C}^2}{2}\tilde{\eta}_u^2&\leq-\frac{\widehat{\mu}_u}{8M}\min\left\{1, \frac{\widehat{\mu}_u}{4ML\tilde{C}^2}\right\}\\
        &\leq-\frac{\bar{\mu}}{8M}\min\left\{1, \frac{\bar{\mu}}{4ML\tilde{C}^2}\right\}\\
        &\eqqcolon\zeta-1 ,
    \end{align}
    where $\bar{\mu}$ is as in \eqref{prop: sublevel coef convergence}. Set $\tilde{u}_{+}=u+ \tilde{\eta}_u (\tilde{v}_u-u)$. Then we can write
    \begin{equation*}
        r_J(\tilde{u}_{+})\leq r_J(u)+(\zeta-1)\Phi(u) .
    \end{equation*}
    Since $r_J(u)\leq\Phi(u)$ and $\zeta-1<0$, we conclude
    \begin{equation*}
        r_J(\tilde{u}_{+})\leq r_J(u)+(\zeta-1)r_J(u)=\zeta r_J(u) .
    \end{equation*}
\end{proof}

\begin{algorithm}
\caption{\texttt{LSIStep}}
\label{alg: LSIStep}
\KwIn{Measure $u$}
\KwOut{Measure $u_+$}
$\mathcal{B}_u\leftarrow \texttt{LSI}(u)$\\
$\tilde{v}_u\leftarrow \sum_{x \in \mathcal{B}_u} u(B_{2R}(x)) \delta_x$\\
$\widehat{\mu}_u\leftarrow\vert u(B_{2R}(\widehat{x}_{\textup{LSI}}^u))\vert,\quad$where\quad$\widehat{x}^u_{\text{LSI}}\in\argmax_{x\in\mathcal{B}_u}\left\lbrack\vert p_u(x)\vert-\max_{z\in \A_u\cap B_{2R}(x)}\vert p_u(z)\vert\right\rbrack$\\
$\tilde{\eta}_u\leftarrow\min\left\{1 , \widehat{\mu}_u/\left(16MLC_{K'}^2\left(2M\sqrt{R/\theta}+2MC_{K'}L/(\theta\sqrt{\gamma})+\sqrt{M/\theta}\right)^2\right)\right\}$\\
$\tilde{u}_{+}\leftarrow (1-\tilde{\eta}_u)u+\tilde{\eta}_u\tilde{v}_u$\\
\Return{$\tilde{u}_+$}
\end{algorithm}

\begin{algorithm}
\caption{Lazified \texttt{PDAP} (\texttt{LPDAP})}
\label{alg: local-global}
\SetKw{Goto}{goto}
\SetKw{Line}{line}
\SetKw{And}{and}
\KwIn{Initial iterate $u_0$, initial lazy threshold $\varepsilon_1$ with $r_J(u_0)\leq2M\varepsilon_1$, initial finite-dimensional accuracy $\Psi_1$, constant $C=4LM^2C_K^2$}
$u_{1-}\leftarrow \textup{\texttt{DropStep}}(u_0)$\\
\For{$k=1,2,\cdots$}{
$u_k,w_k\leftarrow\textup{\texttt{CoefficientStep}}(u_{k-}, \Psi_k)$\label{line: coefficient update}\\
$\tilde{u}_{k+}\leftarrow\textup{\texttt{LSIStep}}(u_k)$\\
$\widehat{u}_{k+}, v_k, \e_{k+1}\leftarrow\textup{\texttt{LGCGStep}}(u_k, \e_k, C)$\\
\If{$\varepsilon_{k+1}=0$}{Terminate with~$u_k$ a minimizer of \eqref{eq: problem setting}}
\If{$\Phi^{u_{k-}}(w_k)>\varphi(u_k,v_k)/2$}{
    $\Psi_k\leftarrow\Psi_k/2$\\
    \Goto{\Line{\ref{line: coefficient update}}}\label{line: recompute}
}
$\Psi_{k+1}\leftarrow \Psi_k$\\
$u_{k+}\leftarrow\argmin_{u\in\{\tilde{u}_{k+}, \widehat{u}_{k+}\}}J(u)$\\
$u_{(k+1)-}\leftarrow\textup{\texttt{DropStep}}(u_{k+})$
}
\end{algorithm}

The above results motivate the solution procedure summarized in Algorithm~\ref{alg: local-global}. Let $(u_k)$ be a sequence of iterates generated by Algorithm~\ref{alg: local-global} and assume for now that this sequence is infinite. Set $\A_k\coloneqq\A_{u_k}$ and $p_k\coloneqq p_{u_k}$. This algorithm, as mentioned previously, is both an acceleration of \texttt{LGCG} and a relaxation of \texttt{PDAP}. 

To see the first point, notice that $u_{k+1}$ is always such that $J(u_{k+1})\leq J(\widehat{u}_{k+})$, which means that all of the additional steps in Algorithm~\ref{alg: local-global} can be interpreted as parts of line~\ref{line: improve} of Algorithm~\ref{alg: sublinear parameter-free}. Thus, in particular, noticing that $r_J(u_1)\leq M\e_1$ and using Theorem~\ref{theorem: sublinear}, we can conclude that $r_J(u_k)\rightarrow0$.
 
To see the second point, notice that the \texttt{LSIStep} and \texttt{LGCGStep} can be interpreted as looking for inexact solutions of the maximization problem in \eqref{def:pdapactualsection}, while the \texttt{DropStep} and \texttt{CoefficientStep} are inexact versions of a modified coefficient problem \eqref{eq: positive coefficient update}, which, under the conditions of Lemma~\ref{lemma: output of coefficient step}, also provides inexact solutions of \eqref{eq: low-dimensional problem}.

Algorithm \ref{alg: local-global} computes both lazified \texttt{GCG} directions $v_k$, by Algorithm \ref{alg: lgcg step}, as well as lumped \texttt{LSI} directions $\tilde{v}_k$ by Algorithm~\ref{alg: LSIStep}, choosing the better of the two for the \texttt{GCG} update. We emphasize that the search for locally improved support points, via \texttt{LSI}, is performed before the choice of the lazy \texttt{GCG} direction, since the former is performed locally and provides potential candidates for the latter. 

Since the \texttt{CoefficientStep} does not add any new support points to the iterates, we can use Lemma~\ref{lemma: drop valid} to conclude that $\A_k\subset\Omega_R$ for all $k$ large enough. Also, combining Lemmas~\ref{lemma: drop valid} and \ref{lemma: output of coefficient step} tells us that it holds $\sign(u_k(\{x\}))=\sign(p_k(x))$ for all $x\in\A_k$ for all $k$ large enough. We recall that $\Phi(u_k)$ is only available if an exact call occurs. Hence, we substitute it by a lower estimate $\varphi(u_k,v_k)$, motivated by the construction of lazy \texttt{GCG} steps. On line~\ref{line: recompute}, we restart each iteration with a progressively smaller $\Psi_k$, until $\Phi^{u_{k-}}(w_k) \leq \varphi(u_k,v_k)/2$ is satisfied. At that point, for large $k$, it holds $\Phi^{u_{k-}}(w_k)\geq\Phi_{\A_k}(u_k)$ by Lemma~\ref{lemma: coefficient same Phi} and all conditions of Theorem~\ref{thm:local linear} are satisfied. Thus, it holds $r_J(u_{k+1})\leq\zeta r_J(u_k)$ for all $k$ large enough.

We refer to the aforementioned iteration restarts as \textit{recompute} steps. In order to reflect this additional computational effort, we denote the total number of recompute steps by $s$ and add it as a subscript whenever necessary, e.g. $u_{k,s}$, $\A_{k,s} $, etc., and refer to the successful iterate as $u_k$.

\begin{theorem}\label{theorem:Big O linear}
    Let $u_{k,s}$ be generated by Algorithm~\ref{alg: local-global}. Let $0<\epsilon<\zeta$ be some small positive tolerance and $\zeta$ the constant from Theorem~\ref{thm:local linear}. Then there is a $\check{C}>0$ independent of $\epsilon$ such that $r_J(u_{k,s})\leq\epsilon$ holds whenever $k+s \geq \check{C} \log_\zeta(\epsilon)$.
\end{theorem}
\begin{proof}
We emphasize that the convergence behavior of Algorithm \ref{alg: local-global} does not depend on the particular choice of $\epsilon>0$ but only on its initialization. Note that
\begin{equation*}
    \varphi(u_{k,s}, v_{k,s}) \geq \min \{\Phi(u_{k,s}), M \varepsilon_k\} \geq \min\left\{r_J(u_{k,s}), \frac{1}{2} r_J(u_1), \frac{1}{2}r_J(u_{k,s}) \right\} \geq \frac{1}{2}r_J(u_{k,s})
\end{equation*}
for all occurring $(k,s)$-pairs, where the penultimate and final inequalities follow from $M \varepsilon_k=M \varepsilon_1 \geq r_J(u_1)/2$, if no exact call was encountered up to iteration $k$, and 
\begin{equation*}
    M \varepsilon_k \geq \frac{1}{2} \inf_{\tilde{k} < k} \Phi(u_{\tilde{k}}) \geq \frac{1}{2} \inf_{\tilde{k} < k} r_J(u_{\tilde{k}}))= \frac{r_J(u_{k-1})}{2} \geq \frac{r_J(u_{k,s})}{2},
\end{equation*}
due to monotonicity of Algorithm \ref{alg: local-global}, otherwise. Now, first assume that infinitely many recompute steps occur throughout a run of Algorithm \ref{alg: local-global}, i.e. there is a nondecreasing sequence $(k_i)_{i=1}^\infty\subset\N$ such that $\Phi^{u_{k_i,i-}}(w_{k_i,i})>\varphi(u_{k_i,i}, v_{k_i,i})/2$ for all $i\in \N$. By construction, we then have
\begin{equation*}
    \frac{\Psi_1}{2^{i}}=\Psi_{k_i,i} \geq \Phi^{u_{k_i,i-}}(u_{k_i,i}) >\frac{\varphi(u_{k_i,i}, v_{k_i,i})}{2} \geq  \frac{r_J(u_{k_i,i})}{4} .
\end{equation*}
for all $i$ large enough. Consequently, $r_J(u_{k,s}) \leq \epsilon $ holds whenever
\begin{equation*}
  s\geq\left\lceil\log_\frac{1}{2}\left(\frac{1}{4\Psi_1}\right)+\log_\zeta(\epsilon)\right\rceil \geq\left\lceil\log_\frac{1}{2}\left(\frac{1}{4 \Psi_1}\right)+\log_\frac{1}{2}(\epsilon)\right\rceil,
\end{equation*}
where the second inequality is a consequence of $\zeta>1/2$, which can be directly seen in the proof of Theorem~\ref{thm:local linear}. 

It remains to derive a worst-case estimate on the number of outer iterations, i.e. the number of $k$ updates. Thus, w.l.o.g., assume that Algorithm \ref{alg: local-global} performs infinitely many $k$ updates. Since $(r_J(u_k))$ is monotonically decreasing, and in view of Theorem \ref{thm:local linear}, there is a $\bar{k} \in \N$ independent of $\epsilon$ such that
\begin{equation*}
    r_J(u_k)\leq\frac{r_J(u_{\bar{k}})}{\zeta^{\bar{k}}}\zeta^k \leq \epsilon \quad \text{for all} \quad k \geq \left\lceil\log_\zeta\left(\frac{\zeta^{\bar{k}}}{r_J(u_{\bar{k}})}\right)+\log_\zeta(\epsilon)\right\rceil.
\end{equation*}
Adding both estimates yields the desired claim.
\end{proof}

\begin{corollary}\label{corollary: linear lpdap}
    For all pairs $(k,s)$ with $k+s$ large enough, there holds
    \begin{equation*}
        r_J(u_{k,s})\leq \zeta^{\frac{1}{3}(k+s)} .
    \end{equation*}
\end{corollary}
\begin{proof}
    See Appendix~\ref{app:lpdap}.
\end{proof}

\section{Lazifying point-moving approaches}\label{section: Newton}

Finally, we turn to sliding variants i.e. methods allowing to move support points in addition to coefficient optimization via approximately solving
\begin{equation} \label{def:problemPN}
  \min_{z=(\mathbf{x}, \lambda) \in \mathcal{Z}^N =\Omega^N \times \R^N}   \J_N(z) \coloneqq \left\lbrack F(K \mathcal{U}(z)) +\alpha |\lambda|_{\ell_1} \right \rbrack , \quad \text{where} \quad \mathcal{U}(z)= \sum^N_{j=1} \lambda_j \delta_{x_j} \tag{$\mathcal{P}_N$}
\end{equation}
and $\mathbf{x}=(x^1, \dots, x^N)$, $\lambda=(\lambda^1,\dots,\lambda^N)$ are interpreted as elements of $\R^{dN}$ and $\R^N$, respectively. For this problem, we define the residual of $\J_N$ as
\begin{equation*}
    r_{\J_N}(z)= \J_N(z)-\min_{\tilde{z}\in \mathcal{Z}^N} \J_N(\tilde{z}) \quad \text{for all} \quad z \in \mathcal{Z}^N.
\end{equation*}
Throughout this section, we again silently assume that Assumptions~\ref{ass:functions} and \ref{ass:dualvariable} hold and, for the sake of simplicity, assume that $\kappa \in \C^{2,1}(\Omega, Y) $ and $F$ is twice continuously Fréchet differentiable. As a consequence, $\J_N$ is of class $\mathcal{C}^{2}$ on $ \mathring{\mathcal{Z}}^N=\operatorname{int}(\Omega)^N \times (\R \setminus \{0\})^N $ with Lipschitz-continuous derivatives on compact subsets which can be readily calculated via the chain rule. 

We one again move auxiliary proofs to the Appendix~\ref{app: Newton} to improve readability.

Given a sparse measure $u$, we call $z \in \mathcal{Z}^N $, $N=\# \A_u$, with $\mathcal{U}(z)=u$ a \textit{minimal representer} of $u$, abbreviated by $z= \texttt{MR}(u)$. Note that minimal representers are unique up to suitable permutations of their components. By definition, we have $\J_N(z) \geq J(\mathcal{U}(z)) $ for all $z \in \mathcal{Z}^N$ and $\J_N(z) = J(\mathcal{U}(z))$ for minimal representers. 

For $N \geq\bar{N}$, we readily verify that the set of minimizers of \eqref{def:problemPN} consists of all admissible $\tilde{z} \in \mathcal{Z}^N$ with $\mathcal{U}(\tilde{z})=\bar{u}$. In particular, for $N=\bar{N}$, \eqref{def:problemPN} admits exactly $N$ minimizers which are obtained by permutations of
\begin{equation*}
    \bar{z}=(\bar{\mathbf{x}}, \bar{\lambda}) \in \mathring{\mathcal{Z}}^N \quad \text{with} \quad \bar{\mathbf{x}}=(\bar{x}^1,\dots, \bar{x}^{\bar{N}}),~ \bar{\lambda}=(\bar{\lambda}^1,\dots, \bar{\lambda}^{\bar{N}})
\end{equation*}
where we consider the same numbering as in Assumption \ref{ass:dualvariable}. In this case, we set $\bar{\mathcal{Z}}= \argmin \eqref{def:problemPN}$ and define the distance
\begin{equation*}
    \operatorname{dist}(z, \bar{\mathcal{Z}})\coloneqq\min_{\tilde{z}\in \bar{\mathcal{Z}}} \|z-\tilde{z}\| \quad \text{for all} \quad z \in \mathcal{Z}^N.
\end{equation*}
If $u= \mathcal{U}(z_1)=\mathcal{U}(z_2)$, $z_1,z_2 \in \mathcal{Z}^N$, note that $\operatorname{dist}(z_1, \bar{\mathcal{Z}})=\operatorname{dist}(z_2, \bar{\mathcal{Z}})$.

For $N>\bar{N}$, we similarly conclude that \eqref{def:problemPN} admits infinitely many minimizers. 

Despite its finite-dimensionality, we emphasize that \eqref{def:problemPN} is significantly more challenging than coefficient optimization, first, due to its nonconvexity, caused by the nonlinearity of the kernel $\kappa$, as well as, second, the potentially complicated geometry of $\Omega$. However, for $N=\bar{N}$, \eqref{def:problemPN} satisfies a second-order sufficient optimality condition in its global minimizers.  
\begin{proposition}
    \label{lemma: Hess j positive definite}
   Set $N= \bar{N}$ and let $\bar{z} \in \mathring{\mathcal{Z}}^N $ be a global minimizer of \eqref{def:problemPN}. Then $\nabla \J_N(\bar{z})=0$ and $\nabla^2 \J_N(\bar{z}) $ is positive definite.
\end{proposition}
\begin{proof}
The statement on the gradient follows immediately since $\bar{z}$ is a minimizer of \eqref{def:problemPN} and $\J_N$ is smooth in the vicinity of $\bar{z}$. The definiteness of the Hessian follows by similar arguments as in \cite[Theorem 4.41]{dissertationWalter}. 
\end{proof}
Hence, given a sufficiently close initial guess, \eqref{def:problemPN} can be efficiently solved via Newton-type methods, which is the main idea pursued throughout this section.
Note that this result is not true for $N>\bar{N}$. In this case, for any minimizer $\tilde{z}=(\tilde{\mathbf{x}}, \tilde{\lambda}) \in \mathring{\mathcal{Z}}_N  $ of \eqref{def:problemPN}, the block matrix $\nabla^2_{\lambda \lambda} \J_N(\tilde{z})$, characterized by
\begin{equation*}
    \delta \lambda^\top \nabla^2_{\lambda \lambda} \J_N(\tilde{z}) \delta \nu=( K\mathcal{U}((\tilde{\mathbf{x}}, \delta \lambda)) ,   \nabla^2 F(K \mathcal{U}(\tilde{z}) ) K\mathcal{U}((\tilde{\mathbf{x}}, \delta \nu)) )_Y \quad \text{for all} \quad \delta \lambda, \delta \nu \in \R^N,
\end{equation*} 
is singular. As a consequence, the proposed algorithm will depend on three building blocks:
\begin{enumerate}
    \item An outer loop consisting of \texttt{LGCG} steps, see Algorithm \ref{alg: lgcg step}, approximate coefficient minimization, as well as drop steps to ensure the global convergence of $u_k$ towards $\bar{u}$ as well as a localization of $\A_k$ around $\bar{\A}$. 
    \item Merging steps to eventually identify the correct number of support points. 
    \item An inner loop performing Newton steps on \eqref{def:problemPN} starting from a minimal representer of the current iterate.
\end{enumerate}

In order to avoid getting stuck prematurely inside of the inner loop, we start by deriving local descent properties of Newton's method in the vicinity of a global minimizer $\bar{z}$ of \eqref{def:problemPN}, $N= \bar{N}$, and relate these to the per-iteration guarantees of the \texttt{LGCG} method, see Lemma \ref{lemma: LGCG in Newton setting}. For this purpose, denote by $m$ and $\bar{m}$ lower and upper estimates on the smallest and largest eigenvalue of $\nabla^2\J_{\bar{N}}(\bar{z})^{-1}$, respectively. Note that these are independent of the particular choice of the global minimizer since they are given by permutations of the $\bar{z}$. In particular, for all $z\in\mathcal{Z}^{\bar{N}}$ in a close enough neighborhood of $\bar{\mathcal{Z}}$, the eigenvalues of $\nabla^2\J_{\bar{N}}(z)^{-1}$ are bounded by $m/2$ and $2\bar{m}$.

In the following, we denote
\begin{align*}
    \texttt{Newton}(z)=\begin{cases} z- \nabla^2 \J_N(z)^{-1} \nabla \J_N(z) &,\quad \det(\nabla^2 \J_N(z))\neq 0  \\
    z &,\quad \text{else}
    \end{cases}.
\end{align*}
By Taylor-approximation as well as standard Newton arguments, we conclude the existence of a Radius $0<\nu_0< R$ as well as of a constant $c_{\text{New}}>0$ such that for all minimizers $\bar{z}$ of Problem \eqref{def:problemPN}, there holds $B_{\nu_0}(\bar{z}) \subset \mathring{\mathcal{Z}}^N$ and every $z=(\mathbf{x}, \lambda) \in B_{\nu_0}(\bar{z}) $ together with its Newton update $z_+=(\mathbf{x}_+, \lambda_+)= \texttt{Newton}(z)$ and $u= \mathcal{U}(z)$ satisfy
\begin{equation} \label{eq:Newtondrop}
    |p_u (x^j)| > \alpha-\sigma/2, \quad \sign(p_u (x^j))=\sign(\lambda^j) \quad \text{for all} \quad j \leq N 
\end{equation}
as well as
\begin{equation} \label{eq:propNewton}
    z_+ \in B_{\nu_0}(\bar{z}), \quad \|z_+-\bar{z}\| \leq c_{\text{New}} \|z-\bar{z}\|^2, \quad \frac{1}{2\bar{m}}\Vert z-\bar{z}\Vert^2\leq r_{\J}(z)\leq\frac{2}{m}\Vert z-\bar{z}\Vert^2 ,
\end{equation}
and
\begin{equation} \label{eq:descentprop1}
 |\lambda_+|_{\ell_1}\leq M, \quad    \J_N(z_+)-\J_N(z)\leq-\frac{m}{8}\Vert\nabla\J_N (z)\Vert^2, \quad \Vert\nabla\J_N(z)\Vert^2\geq\frac{1}{\bar{m}}r_{\J_N}(z). 
\end{equation}
In particular, \eqref{eq:propNewton} implies
\begin{equation} \label{eq:quadconvintro}
 r_{J}(\mathcal{U}(z_+)) \leq  r_{\J_N}(z_+) \leq C_{\text{New}} r_{J}(u)^2 \quad \text{where} \quad C_{\text{New}} = \frac{8 c_{\text{New}}\bar{m}^2}{m}.
\end{equation}
Moreover, given a tolerance $\varepsilon>0$, $u= \mathcal{U}(z)$, and the usual value of $C$, denote by $(u_+,v, \varepsilon_+ )$ the corresponding \texttt{LGCG}-step, i.e. the output of Algorithm \ref{alg: lgcg step}. Invoking Lemma \ref{lemma: LGCG in Newton setting} yields
\begin{equation} \label{eq:NewtonconnectoGCG}
     r_{\mathcal{J}_N}(z) \geq    J(u)-J(u_+)\geq
        \begin{cases}
            \frac{M^2\e_+^2}{2C} &,\quad M\e_+\leq C\\
            \frac{2M\e_+ -C}{2} &,\quad\text{else}
        \end{cases}.  
\end{equation}
Motivated by these estimates, and now for arbitrary $N \in \N$, we accept the Newton step if 
\begin{equation} \label{eq:acceptanceNewton1}
 z_+ \in \mathcal{Z}^N, \quad  |\lambda_+|_{\ell_1}\leq M, \quad    \J_N(z_+)-\J_N(z)\leq-\frac{m}{8}\Vert\nabla\J_N (z)\Vert^2,
\end{equation}
as well as
\begin{equation} \label{eq:acceptanceNewton2}
    \Vert\nabla\J_N(z)\Vert^2\geq \begin{cases}
            \frac{M^2\e^2}{2C\bar{m}} &,\quad M\e\leq C\\
            \frac{2M\e -C}{2\bar{m}} &,\quad\text{else}
        \end{cases}, 
\end{equation}
where $m$ and $\bar{m}$ are treated as hyperparameters and $\e >0$ will be adapted throughout the iterations to avoid unnecessary \texttt{LGCG} steps. Finally, Algorithm \ref{alg: local merging} describes the aforementioned merging procedure relying on the radius parameter $R>0$. The overall procedure is summarized in Algorithm \ref{alg: sublinear with parameters and newton}. To establish its convergence, we will rely on the following observation concerning the combination of drop and local merging steps.

\begin{lemma} \label{lem:genericsparse}
    Consider sequences $\seq{u_n}$, $\seq{u^{\textup{drop}}_n}$ and $\seq{\tilde{u}_n}$ with $u_n \weakstar \bar{u}$, $u^{\textup{drop}}_n=\textup{\texttt{DropStep}}(u_n)$ and
    \begin{equation*}
        J(\tilde{u}_n) \leq J(u^{\textup{drop}}_n), \quad \A_{\tilde{u}_n} \subset \A_{u^{\textup{drop}}_n}, \quad \sign(\tilde{u}_n(\{x\}))=\sign(u^{\textup{drop}}_n(\{x\})) \quad \text{for all} \quad x \in \A_{\tilde{u}_n}.
    \end{equation*}
    Set $u^{\textup{lump}}_n=\textup{\texttt{LM}}(\tilde{u}_{n})$. Then there  is $n(\nu_0)\in\N$ such that we have $\# \A_{u^{\textup{lump}}_n}=\bar{N} $ for all $n \geq n(\nu_0)$ and every minimal representer $z^{\textup{lump}}_n=\textup{\texttt{MR}}(u^{\textup{lump}}_n)$ satisfies $\operatorname{dist}(z^{\textup{lump}}_n, \bar{\mathcal{Z}})< \nu_0$.
\end{lemma}
\begin{proof}
 See Appendix~\ref{app: Newton}.
\end{proof}

The next lemma shows that the inner loop of Algorithm \ref{alg: sublinear with parameters and newton} yields quadratic convergence.

\begin{lemma} \label{lem:insideofball}
Let $u_{k,s}$ and $z_{k,s}=\textup{\texttt{MR}}(u_{k,s})$ be generated by Algorithm \ref{alg: sublinear with parameters and newton} and assume that $k$ and $s$ are such that there holds
\begin{equation*}
\#\A_{u_{k,s}}=\bar{N}, \quad \operatorname{dist}(z_{k,s}, \bar{\mathcal{Z}})< \nu_0.
\end{equation*}
Then $u_{k,s+1}$ is well-defined, there holds $u_{k,s+1}=u^{\text{New}}_{k,s}$, $\#\A_{u_{k,s+1}}=\bar{N}$ and every minimal representer $z_{k,s+1}=\textup{\texttt{MR}}(u_{k,s+1})$ satisfies $\operatorname{dist}(z_{k,s+1}, \bar{\mathcal{Z}})< \nu_0$. Moreover, there holds
\begin{equation*}
     r_{J}(u_{k,s+1}) \leq C_{\text{New}} \, r_{J}(u_{k,s})^2
\end{equation*}
\end{lemma}
\begin{proof}
    See Appendix~\ref{app: Newton}.
\end{proof}

Iterating this argument leads to the following corollary.

\begin{corollary} \label{coroll:stuckininner}
Assume that $u_{k,s}$ and $z_{k,s}=\textup{\texttt{MR}}(u_{k,s})$ are generated by Algorithm \ref{alg: sublinear with parameters and newton} and satisfy the assumptions of Lemma \ref{lem:insideofball}. Then Algorithm \ref{alg: sublinear with parameters and newton} does not exit the inner $\textup{\texttt{for}}$ loop in iteration $k$ and yields a sequence $\seq{u_{k,s+n}}$ such that
\begin{equation*}
     r_{J}(u_{k,s+n+1}) \leq C_{\text{New}} \, r_{J}(u_{k,s+n})^2, \quad \# \A_{u_{k,s+n}}= \bar{N}.    
\end{equation*}
for all $n \geq 0$.
\end{corollary}

We are now ready to prove that Algorithm \ref{alg: sublinear with parameters and newton} eventually recovers the correct number of support points and exhibits an asymptotic quadratic rate of convergence. 

\begin{theorem} \label{thm:quadconvergence}
    Let $\seq{u_{k,s}}$ be generated by Algorithm \ref{alg: sublinear with parameters and newton}. Then there are a $\bar{k}$ as well as an $\bar{s}$ such that Algorithm \ref{alg: sublinear with parameters and newton} does not exit the inner $\textup{\texttt{for}}$ loop in iteration $\bar{k}$ and satisfies
\begin{equation*}
     r_{J}(u_{\bar k,\bar{s}+n+1}) \leq C_{\text{New}} \,r_{J}(u_{\bar k,\bar{s}+n})^2, \quad \# \A_{u_{\bar k,\bar{s}+n}}= \bar{N}
\end{equation*}
for all $n \geq 0$.
\end{theorem}
\begin{proof}
We will show that $u_{k,s}$ and $z_{k,s}=\textup{\texttt{MR}}(u_{k,s})$ satisfy the assumptions of Lemma \ref{lem:insideofball} after a finite number of outer and inner iterations, the claimed convergence results then follow Corollary \ref{coroll:stuckininner}. For this purpose, we split the discussion into two parts:

First, assume that there is an outer iteration number $\bar{k}$ such that Algorithm \ref{alg: sublinear with parameters and newton} does not exit the inner $\textup{\texttt{for}}$ loop in iteration $\bar k$. For the sake of readability, we drop the index $ \bar k$ and consider the sequence $u_s=u_{\bar k,s}$. We start by noting that, for $s \mod S=0$, we have $\#\A_{u^{\text{drop}}_s}  \leq\# \A_{u^{\text{New}}_{s}}$ with strict inequality iff $u^{\text{drop}}_s \neq u^{\text{New}}_{s} $. Mutatis mutandis, the same holds true for the local merging step on line~\ref{line: local merging}. Since the number of support points does not increase in the remainder of the \texttt{for} loop, we can thus assume that $u_{s+1}=u^{\text{New}}_s$ holds for all $s$ large enough. Notice that, by Lemma~\ref{lemma: LGCG in Newton setting} as well as by the $\e$ update procedure in lines~\ref{line: epsilon update outer} and \ref{line: epsilon update inner} of Algorithm~\ref{alg: sublinear with parameters and newton}, it holds $r_{J}(u_{s}) \leq 2M \e_s$ for all $s$. Furthermore, it holds
\begin{align*}
\J_N(z_{s+1})-\J_N(z_s)\leq-\frac{m}{8}\Vert\nabla\J_N (z_s)\Vert^2, \quad \|\nabla \mathcal{J}_N(z_{s})\|^2 \geq \begin{cases}
            \frac{M^2\e^2_{s+1}}{2C\bar{m}} &,\quad M\e_{s+1}\leq C\\
            \frac{2M\e_{s+1} -C}{2\bar{m}} &,\quad\text{else}
        \end{cases}   
\end{align*}
for all $s$ large enough and we can conclude
\begin{equation*}
    \lim_{s \rightarrow \infty} \|\nabla \mathcal{J}_N(z_s)\| = \lim_{s\rightarrow\infty}\e_s = \lim_{s \rightarrow \infty} r_J(u_s)=0.
\end{equation*}
In particular, we have $u_s \weakstar \bar{u}$. Applying Lemma \ref{lem:genericsparse} to $u_s$, the desired convergence statement then follows from Corollary \ref{coroll:stuckininner}.

Second, assume that, for every $k$, Algorithm \ref{alg: sublinear with parameters and newton} leaves the inner loop after a finite number of steps. Noting that
\begin{equation*}
    J(u_{k+1})\leq J(u_k^\text{coef}) \leq J(u^{\text{drop}}_k) \leq J(u^{\text{GCG}}_k) \leq J(u_k),
\end{equation*}
we conclude $u_k \rightharpoonup^* \bar{u}$ from Theorem \ref{theorem: sublinear}. Hence, again invoking Lemma \ref{lem:genericsparse} (via Lemma~\ref{lemma: output of coefficient step} applied to $u_k^\text{coef}$), we conclude that Corollary \ref{coroll:stuckininner} is applicable to $u_{\bar{k},1}$ and $z_{\bar{k},1}$ for some $\bar{k}$, yielding a contradiction to the assumption that the inner loop is left after finitely many steps.
\end{proof}

Note that, although the \texttt{CoefficientStep} is not needed for the convergence of Algorithm~\ref{alg: sublinear with parameters and newton}, in practice, it greatly accelerates the initial warm-up phase.

\begin{algorithm}
\caption{Local Merging (\texttt{LM})}
\label{alg: local merging}
\KwIn{Sparse measure $u$}
\KwOut{Merged measure $u^{\textup{lump}}$}
$\A\leftarrow \A_u$, $u^{\text{lump}} \leftarrow 0$\\
\While{$\A \neq \emptyset$}{
 choose $x \in \argmax_{x \in \A} |p_u(x)|$\\
 $u^{\text{lump}} \leftarrow u^{\text{lump}}+ u(B_{2R}(x)) \delta_x $ \\
    $\A \leftarrow \A \setminus B_{2R}(x)$
}
\Return{$u^{\textup{lump}}$}
\end{algorithm}

\begin{algorithm}
\caption{Newton Lazified Generalized Conditional Gradient (\texttt{NLGCG})}
\label{alg: sublinear with parameters and newton}
\SetKw{Break}{break}
\KwIn{Initial iterate $u_1$, initial lazy threshold $\varepsilon_1$ with $r_J(u_1)\leq2M\varepsilon_1$, initial finite-dimensional accuracy $\Psi_1$, merging frequency $S$, constant $C=4LM^2C_K^2$}
\For{$k=1,2,\dots$}{
    $u_{k}^{\text{GCG}}, v_k, \e_{k+1}\leftarrow\textup{\texttt{LGCGStep}}(u_k,\e_k,C)$\\
    \If{$\varepsilon_{k+1}=0$}{Terminate with~$u_k$ a minimizer of \eqref{eq: problem setting}}
    $u^{\text{drop}}_{k}\leftarrow\textup{\texttt{DropStep}}(u_{k}^{\text{GCG}})$\\
    $u^{\textup{coef}}_k,w_k\leftarrow\textup{\texttt{CoefficientStep}}(u^{\text{drop}}_{k}, \Psi_k)$\\
    $u_{k,1}\leftarrow\textup{\texttt{LM}}(u^{\textup{coef}}_k)$\\
    $\e_{k,1}\leftarrow\e_{k+1}+\frac{J(u_{k,1})-J(u^{\textup{coef}}_k)}{2M}$\label{line: epsilon update outer}\\
    \For{$s=1,2,\dots$\label{line: for loop}}{
        $z_{k,s} \leftarrow \texttt{MR}(u_{k,s})$,\quad$z^{\text{New}}_{k,s}\leftarrow\textup{\texttt{Newton}}(z_{k,s})$\\
        $u^{\text{GCG}}_{k,s} \leftarrow u_{k,s}$,\quad$u^{\text{New}}_{k,s} \leftarrow \mathcal{U}(z^{\text{New}}_{k,s}) $ \\
        \eIf{$z_{k,s}$ does not satisfy \eqref{eq:acceptanceNewton2}, $\varepsilon=\varepsilon_{k,s}$}
        {
        $u^{\text{GCG}}_{k,s}, v_{k,s}, \e_{k,s+1}\leftarrow\textup{\texttt{LGCGStep}}(u_{k,s},\e_{k,s},C)$\\
        \If{$\varepsilon_{k,s+1}=0$}{Terminate with~$u_{k,s}$ a minimizer of \eqref{eq: problem setting}}
        \If{$z_{k,s}$ does not satisfy \eqref{eq:acceptanceNewton2}, $\varepsilon=\varepsilon_{k,s+1}$}
        {\Break{} line \ref{line: for loop}}
        }
        {
        $\e_{k,s+1}\leftarrow \e_{k,s}$
        }
        \If{$z_{k,s},z^{\textup{New}}_{k,s}$ do not satisfy \eqref{eq:acceptanceNewton1}}{
        \Break{} line \ref{line: for loop}}
        \eIf{ $s \mod S=0$}{$u^{\text{drop}}_{k,s}\leftarrow\textup{\texttt{DropStep}}(u^{\text{New}}_{k,s})$ \\$u_{k,s+1} \leftarrow \texttt{LM}(u^{\text{drop}}_{k,s}) $\label{line: local merging} \\$\e_{k,s+1}\leftarrow\e_{k,s+1}+\frac{J(u_{k,s+1})-J(u^{\text{drop}}_{k,s})}{2M}$\label{line: epsilon update inner} }
        {
        $u_{k,s+1} \leftarrow u^{\text{New}}_{k,s}$\\}
    }
    Choose $u_{k+1}\in\argmin_{u\in\{u_{k}^{\text{coef}},u_{k,1},u^{\text{New}}_{k,s},u^{\text{GCG}}_{k,s}\}}J(u)$\label{line: choice of measure}\\
    $\Psi_{k+1}\leftarrow\Psi_k/2$
}
\end{algorithm}

\begin{remark} \label{rem:othersubroutines}
We emphasize that the Newton step in the inner loop of Algorithm \ref{alg: sublinear with parameters and newton} can be replaced by any other method suitable for \eqref{def:problemPN} and which guarantees local convergence in the vicinity of stationary points with positive definite Hessian, e.g. damped or Quasi-Newton methods. In this case, Theorem \ref{thm:quadconvergence} can be adapted, mutatis mutandis, by replacing the quadratic decrease with an estimate reflecting the asymptotic convergence rate of the respective method. Moreover, we point out that the prescribed per-iteration descent, i.e. the third requirement in \eqref{eq:acceptanceNewton1}, is only needed to ensure that $\nabla \mathcal{J}_N(z_{k,s}) \rightarrow 0$ holds if Algorithm \ref{alg: sublinear with parameters and newton} does not exit the inner loop. Consequently, it can be dropped if the latter is ensured by the particular choice of the employed method, e.g. by means of globalization.
\end{remark}

\section{Numerical experiments}\label{section: numerics}

We close the manuscript with two numerical experiments demonstrating the advantages of the lazy approach towards greedy point insertion and verifying our theoretical results. The experiments are run in \texttt{Python 3.10}, with parallelization and scientific computing functionalities provided by the \texttt{numpy} module. All the code related to these experiments can be found in our GitHub repository\footnote{\url{https://github.com/arsen-hnatiuk/lazified-pdap}}. The tests are executed on a Dell OptiPlex 7060 desktop computer with an Intel Core i7-8700 CPU and 16GB of RAM.

For each problem we run the original \texttt{PDAP} as described in \eqref{def:pdapactualsection}, the proposed lazified version \texttt{LPDAP}, Algorithm~\ref{alg: local-global}, as well as the proposed sliding variant \texttt{NLGCG}, Algorithm~\ref{alg: sublinear with parameters and newton}, all starting from the zero measure. The hyperparameters for the latter two are heuristically chosen separately for each example. Moreover, since all of the considered algorithms guarantee descent, we dynamically update $M$, noting that
\begin{equation*}
    \mnorm{u_k} \leq M_k \coloneqq J(u_k)/\beta
\end{equation*}
since $F$ is nonnegative.

In all experiments, the global search for lazy and exact solutions of the maximization problem in  \eqref{def:pdapactualsection} inside the \texttt{LGCGStep} Algorithm~\ref{alg: lgcg step} is implemented using a Newton's method, initiated at the nodes of an equally spaced grid over $\Omega$, as well as the support of the current iterate. The search for local improved support points in the \texttt{LSI} Algorithm~\ref{alg: local improver} is performed using a Newton's method initiated on a subset of the support of the current iterate. In all cases, a maximum of 5 iterations of the Newton's method are performed per initialization. In the case of the \texttt{LSI}, if an iterate does not fulfill the defining properties of an improved point $x_{LSI}$ after 5 iterations, we conclude that such a point cannot be found. The \texttt{CoefficientStep} Algorithm~\ref{alg: low-dimensional problem} is implemented using a semismooth Newton method based on the normal map, as described in \cite{dissertationMilzarek,dissertationPieper}. Values below $10^{-12}$ are treated as zero.

In order to measure the progress of the considered algorithms, we use the dual gap $\Phi(u_k)$ for \texttt{PDAP}, the estimate $\varphi(u_k,v_k)$ of $\Phi(u_k)$ for \texttt{LPDAP}, and the estimate $2M\e_k$ of the residual $r_J$ for \texttt{NLGCG}. In all cases, we run the specific algorithm until we reach an iterate bringing the associated quantity below the tolerance $10^{-12}$. Finally, we approximate the residual by $r_J(u) \approx J(u)-J(\tilde{u})$ where  $\tilde{u}$ is found by running \texttt{PDAP} until we have $r_J(\tilde{u}) \leq \Phi(\tilde{u})\leq 10^{-14} $. For all three algorithms, we compare the evolution of $r_j(u_k)$ as well as the support size $\# \A_k$ throughout the iterations as well as w.r.t. computational time. In this regard, we consider the total number of inner and outer iterations for \texttt{LPDAP} and \texttt{NLGCG}. Moreover, for the two lazy algorithms we also report on the update of the lazy threshold $\varepsilon_k$ and the number of lazy and exact calls, respectively.

\subsection{Source identification}
As a first experiment, we consider the identification of the initial condition of a free-space heat equation from scarce observations of the associated state at a given time $t>0$, similar to \cite[Section~4.1]{BrediesCarioniFanzonWalter} and  \cite{LeykekhmanDmitriy2020Naos}. In our setting, we describe this by considering $\Omega=[0,1]\times[0,1]$ as well as $F$ and $Ku \in Y=\R^{16}$ via
\begin{equation*}
    F(y)= \frac{1}{2} \|y-y^\dagger\|^2, \quad [Ku]_i=\int_\Omega\frac{1}{4t\pi}e^{-\frac{\Vert x-x_i\Vert^2}{4t}}du(x),\quad i=1,\dots,16,
\end{equation*}
where $x_1,\dots, x_{16}$ are the nodes of a uniform $4\times4$ grid over $\Omega$ and $y^\dagger=Ku^\dagger$ is the observation for the ground truth 
\begin{equation*}
    u^\dagger=1\delta_{(0.28,0.71)}-0.7\delta_{(0.51,0.27)}+0.8\delta_{(0.71,0.53)},
\end{equation*}
which we try to identify.
\begin{table}[ht]
    \centering
    \begin{tabular}{||c c c c c c c c c c c||}
        \hline
        $\alpha$ & $t$ & $\theta$ & $\gamma$ & $\sigma$ & $L$ & $R$ & $m$ & $\bar{m}$ & $C_K$ & $C_{K'}$\\
        \hline 
        $0.1$ & $0.025$ & $0.1$ & $1$ & $0.002$ & $1$ & $0.01$ & $0.001$ & $0.1$ & $6.26$ & $27.13$\\
        \hline
    \end{tabular}
    \caption{Parameters used in the initial source location experiment.}
    \label{table: parameters}
\end{table}

\begin{figure}[ht]
\centering
\begin{subfigure}[b]{0.49\textwidth}
  \centering
  \includegraphics[width=\linewidth]{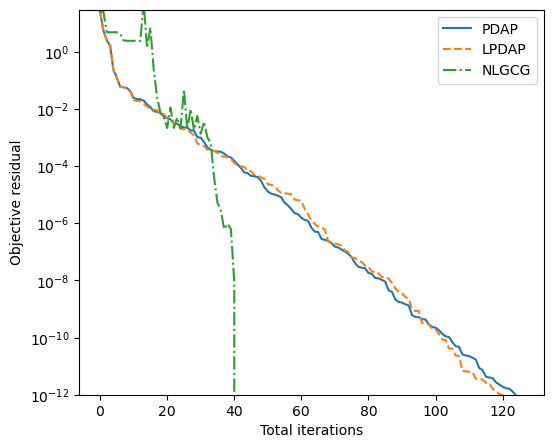}
\end{subfigure}
\hfill
\begin{subfigure}[b]{0.49\textwidth}
  \centering
  \includegraphics[width=\linewidth]{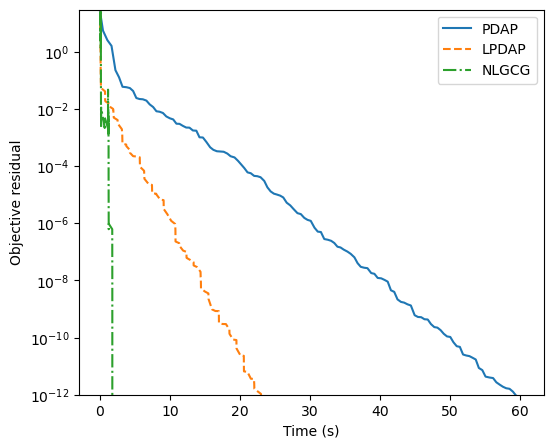}
\end{subfigure}
\caption{Convergence behavior of the tested algorithms.}
\label{fig: heat convergence}
\end{figure}

The hyperparameters for \texttt{LPDAP} and \texttt{NLGCG} are listed in Table \ref{table: parameters}. As predicted, both, \texttt{PDAP} and \texttt{LPDAP} exhibit a linear convergence behavior, while we observe a vastly improved convergence rate for \texttt{NLGCG}. For the latter, the occasional increase in the residual is both  caused by merging as well as bad Newton steps leading to ascent and thus a break of the inner loop. Subsequently, these are compensated by line~\ref{line: choice of measure} of Algorithm~\ref{alg: sublinear with parameters and newton}. The convergence behavior of \texttt{NLGCG} is further illustrated in Figure~\ref{fig: heat local routine} where inner iterations are shaded. Upon entering the first two inner loops, we observe a stagnation in the Newton process, indicating convergence towards a stationary point of \eqref{def:problemPN} with $N \neq \bar{N}$, which avoids getting stuck due to the globalization strategy in \eqref{eq:acceptanceNewton2}. Similar observations can be made in subsequent runs of the inner loop, leading to repeated (lazy) point insertions which manifest as a stepfunction in the support size plot, Figure \ref{fig: support in iterations}. However, asymptotically, excessive points are removed via the merging step and \texttt{NLGCG} enters the asymptotic quadratic convergence regime with $\#\A_k=3$, coinciding with the number of global extrema of $|\bar{p}|$ as predicted by the theory. In comparison, both \texttt{PDAP} and \texttt{LPDAP} suffer from clustering due to their lack of point moving and severely overestimates the size of the optimal support, see also Figure~\ref{fig: clustering}. This observation is most pronounced for \texttt{LPDAP} since exact coefficient minimization also helps to sparsify the weights $\lambda^k$, leading to a more aggressive point removal once $\A_k$ is updated. 

Comparing \texttt{PDAP} and \texttt{LPDAP} directly in Figure \ref{fig: heat convergence}, we point out that the rate of both algorithms is almost identical while the plot associated to \texttt{LPDAP} also includes recompute steps (around 30), i.e. \texttt{LPDAP} performs fewer outer iterations. We attribute this beneficial performance to the local \texttt{LSI}-update which potentially adds several new points instead of a single one to the active set. The benefits of including the lazy paradigm becomes most evident once we compare the computational time of the three methods. Indeed, while \texttt{PDAP} and \texttt{LPDAP} require almost the same amount of \texttt{LGCG} steps, the latter predominantly performs lazy steps, with exact calls occurring roughly on every third iteration, see Table \ref{table: heat lazy steps} and Figure \ref{fig: heat epsilon vs iterations}. Given that lazy calls are significantly cheaper than exact updates, this reduces the overall computational time by a factor of three, see Figure \ref{fig: heat convergence}. Concerning \texttt{NLGCG}, note that \texttt{LGCG} steps only occur in a small fraction of steps, i.e. the method mainly performs cheaper Newton updates. Moreover, similar to \texttt{LPDAP}, most \texttt{LGCG} steps are lazy, leading to convergence in a few seconds. Finally, exact steps only represent around $1/10$ of the overall number of iterations and are, in most cases, required for the verification of~\eqref{eq:acceptanceNewton2}. 



\begin{figure}[ht]
    \centering
    \includegraphics[width=1\textwidth]{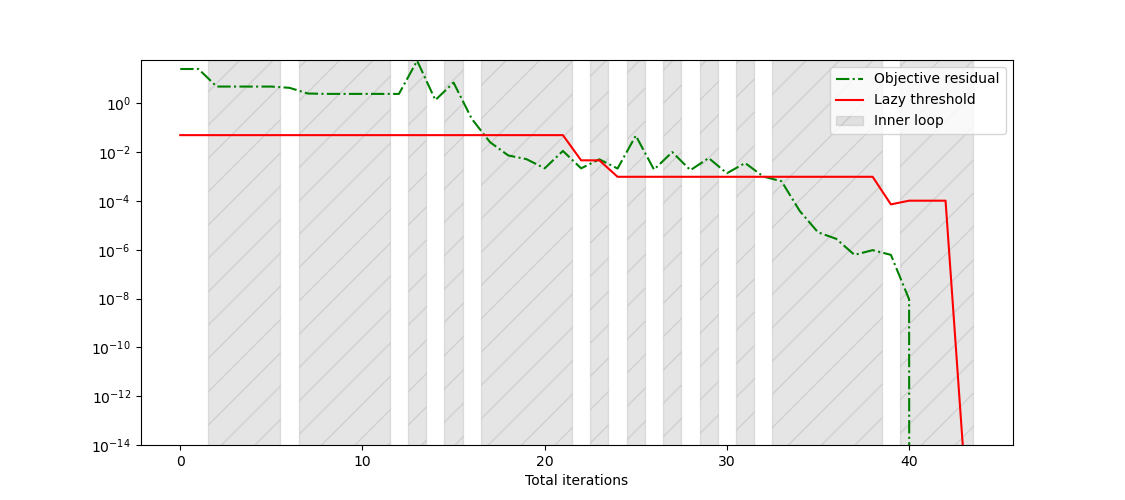}
    \caption{Progression of the \texttt{NLGCG} algorithm, dashed areas correspond to iterations within the inner loop}
    \label{fig: heat local routine}
\end{figure}

\begin{table}[ht]
    \centering
    \begin{tabular}{||c | c c||}
        \hline
        Algorithm & Lazy Calls & Exact Calls\\
        \hline
        \texttt{PDAP} & 0 & 127\\
        \texttt{LPDAP} & 80 & 43\\
        \texttt{NLGCG} & 11 & 4\\
        \hline
    \end{tabular}
    \caption{Number of lazy and exact calls performed by \texttt{LGCGStep} in every algorithm.}
    \label{table: heat lazy steps}
\end{table}

\begin{figure}[ht]
\centering
\begin{subfigure}[b]{0.49\textwidth}
  \centering
  \includegraphics[width=\linewidth]{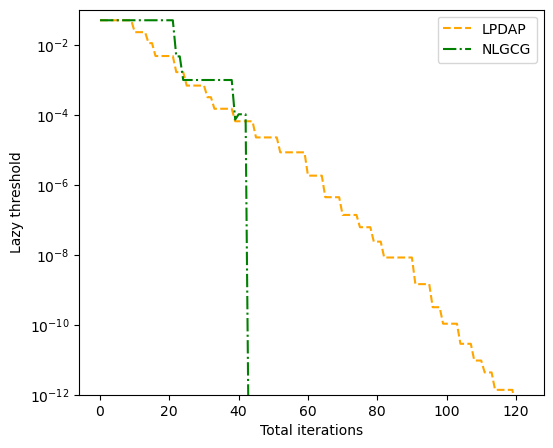}
  \caption{Evolution of the lazy threshold $\e_k$.}
  \label{fig: heat epsilon vs iterations}
\end{subfigure}
\hfill
\begin{subfigure}[b]{0.49\textwidth}
  \centering
  \includegraphics[width=\linewidth]{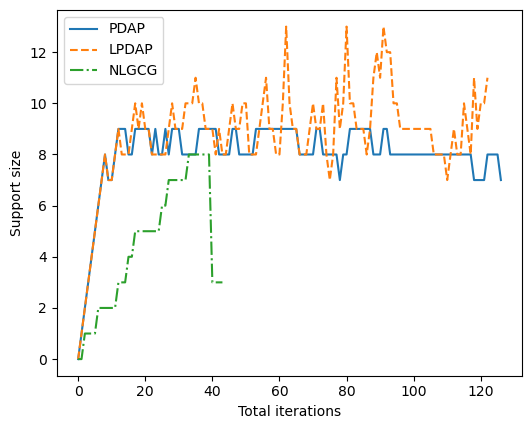}
  \caption{Support size in each iteration.}
  \label{fig: support in iterations}
\end{subfigure}
\caption{Lazy threshold $\e_k$ and support size for each algorithm.}
\end{figure}


\begin{figure}[ht]
\centering
\begin{subfigure}[b]{0.49\textwidth}
  \centering
  \includegraphics[width=\linewidth]{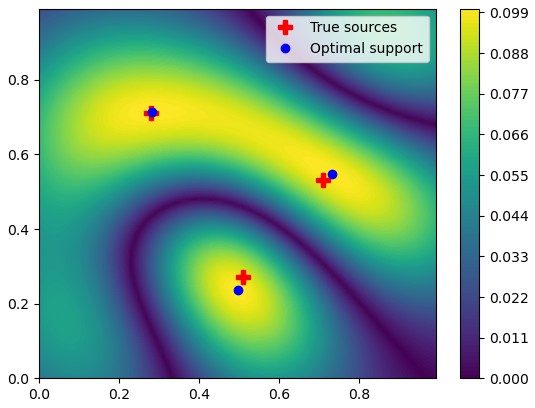}
  \caption{Contour plot of the optimal dual variable $\vert\bar{p}\vert$. Crosses represent the support of the true initial distribution $u^\dagger$ and dots represent the support of the optimal solution $\bar{u}$. Notice that $\vert\bar{p}\vert$ takes its maximum value $\cnorm{\bar{p}}=\alpha$ in the support of $\bar{u}$.}
  \label{fig: optimal dual certificate}
\end{subfigure}
\hfill
\begin{subfigure}[b]{0.49\textwidth}
  \centering
  \includegraphics[width=\linewidth]{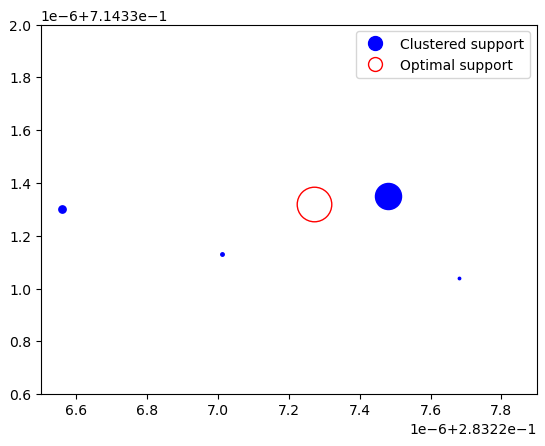}
  \caption{Zoomed-in view of one of the optimal support points (hollow dot). The full dots are the support of an iterate generated by \texttt{PDAP}. The sizes of the dots corresponds to the measure weights. Notice that the scale is of order $10^{-6}$.}
  \label{fig: clustering}
\end{subfigure}
\caption{Behavior of optimal support points}
\label{fig: optimal support}
\end{figure}


\subsection{Signal processing} \label{subsec:signal}
As a second experiment, we consider the recovery of source frequencies from an intercepted signal. The minimization problem itself is similar to the last example using the same $F$ but considering higher-dimensional observations. More in detail, we discretize the time interval $[0,1]$ into $n=120$ equidistant time points $t_i$ and set
\begin{equation*}
    [Ku]_i=\int_\Omega \sin(2\pi t_ix)~\dd u(x) ,\quad i=1,\cdots,120.
\end{equation*}
Concerning the frequency range, we choose $\Omega=[0,60]$. The measurements $y^\dagger \in \R^{120}$ are once again obtained as $y^\dagger=Ku^\dagger$, where
\begin{equation*}
u^\dagger=-1\delta_{3.125}+0.7\delta_7+0.5\delta_{\sqrt{179}}.
\end{equation*}
The resulting signal is illustrated in Figure~\ref{fig: input signal}, while the values of the required hyperparameters are found in Table \ref{table: parameters 2}. As in the previous example, we compare the three algorithms regarding their convergence and computation time, see Figure~\ref{fig: signal convergence}, the evolution of the support size, Figure~\ref{fig: signal support in iterations}, as well as the updates of the lazy threshold, Figure \ref{fig: signal epsilon vs iterations}, in \texttt{LPDAP} and \texttt{NLGCG}, respectively.  

\begin{figure}[ht]
\centering
\begin{subfigure}[b]{.32\textwidth}
  \centering
  \includegraphics[width=\linewidth]{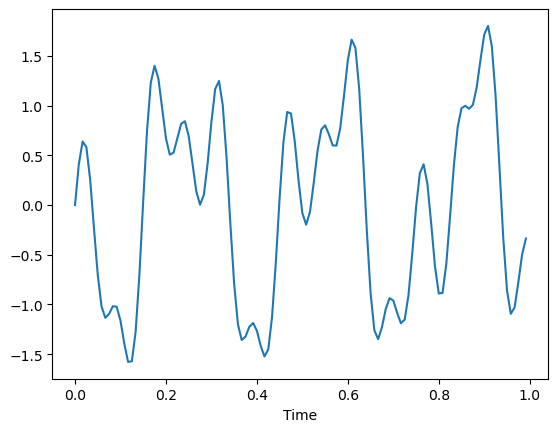}
  \caption{Input signal $y^\dagger$ generated from the true frequency distribution $u^\dagger$ and recorded at $120$ equidistant time points.}
  \label{fig: input signal}
\end{subfigure}
\hfill
\begin{subfigure}[b]{.32\textwidth}
  \centering
  \includegraphics[width=\linewidth]{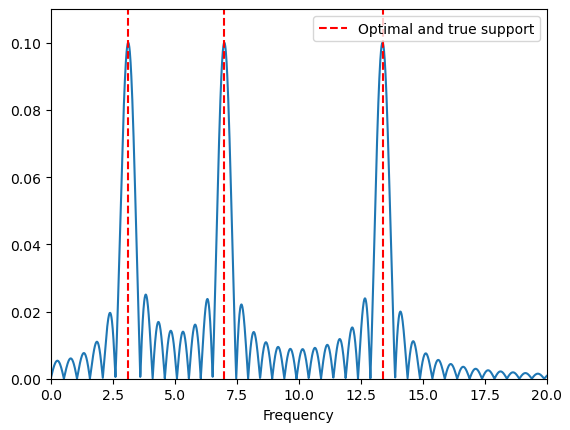}
  \caption{The optimal absolute value dual variable with the locations of the optimal and true support points, restricted to $[0,20]$.}
  \label{fig: signal dual certificate}
\end{subfigure}
\hfill
\begin{subfigure}[b]{.32\textwidth}
  \centering
  \includegraphics[width=\linewidth]{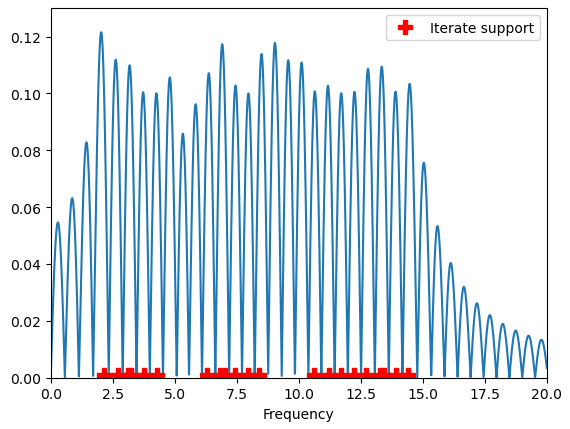}
  \caption{Absolute value dual variable of an intermediate \texttt{LPDAP} iterate with support locations, restricted to $[0,20]$.}
  \label{fig: iterate dual certificate}
\end{subfigure}
\caption{Input signal and dual variable}
\label{fig: optimal support signal}
\end{figure}

\begin{table}[ht]
    \centering
    \begin{tabular}{||c c c c c c c c c c||}
        \hline
        $\alpha$ & $\theta$ & $\gamma$ & $\sigma$ & $L$ & $R$ & $m$ & $\bar{m}$ & $C_K$ & $C_{K'}$\\
        \hline
        $0.1$ & $0.1$ & $1$ & $0.05$ & $1$ & $0.1$ & $0.001$ & $0.1$ & $8.44$ & $39.49$\\
        \hline
    \end{tabular}
    \caption{Parameters used in the signal processing experiment.}
    \label{table: parameters 2}
\end{table}
Both, qualitatively and quantitatively, we can make similar observations as in the last example. Focusing on the differences, note that the gain in computation time by lazifying \texttt{PDAP} is marginal. We attribute this, on the one hand, to the smaller spatial dimension, 1D vs. 2D, which facilitates the calculation of exact \texttt{LGCG} updates. On the other, we again observe a severe overestimation of the optimal support size by \texttt{LPDAP} due to clustering phenomena leading to ill-conditioned coefficient minimization problems and thus increased computation times, see Figure \ref{fig: iterate dual certificate}. Finally, concerning the lazy threshold, we again observe uniformly distributed updates in the case of \texttt{LPDAP}, while \texttt{NLGCG} does not update $\varepsilon_k$ until it enters the asymptotic, quadratic convergence regime, see Figure \ref{fig: signal epsilon vs iterations} and \ref{fig: signal local routine}, respectively.  

\begin{figure}[ht]
\centering
\begin{subfigure}[b]{0.49\textwidth}
  \centering
  \includegraphics[width=\linewidth]{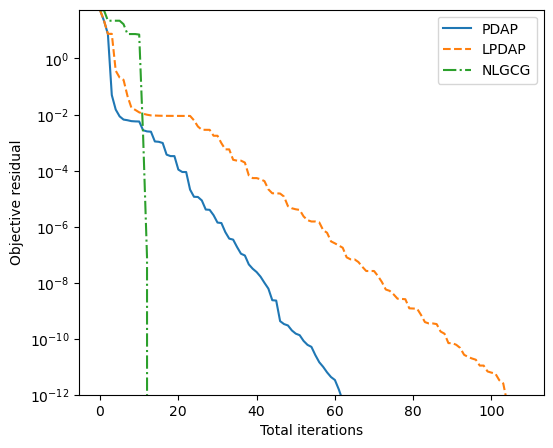}
\end{subfigure}
\hfill
\begin{subfigure}[b]{0.49\textwidth}
  \centering
  \includegraphics[width=\linewidth]{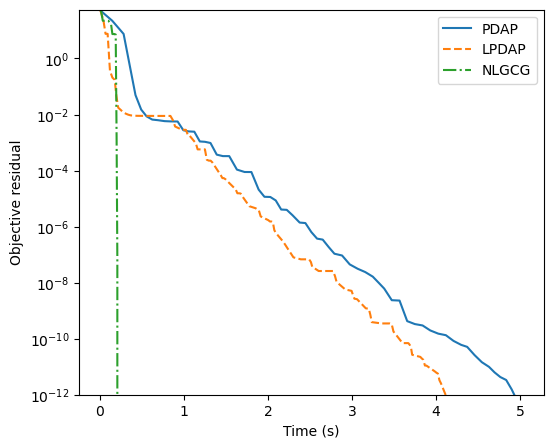}
\end{subfigure}
\caption{Convergence behavior of the tested algorithms.}
\label{fig: signal convergence}
\end{figure}


\begin{table}[ht]
    \centering
    \begin{tabular}{||c | c c||}
        \hline
        Algorithm & Lazy Calls & Exact Calls\\
        \hline
        \texttt{PDAP} & 0 & 64\\
        \texttt{LPDAP} & 79 & 30\\
        \texttt{NLGCG} & 5 & 2\\
        \hline
    \end{tabular}
    \caption{Number of lazy and exact calls performed by \texttt{LGCGStep} in every algorithm.}
    \label{table: signal lazy steps}
\end{table}

\begin{figure}[ht]
\centering
\begin{subfigure}[b]{0.49\textwidth}
  \centering
  \includegraphics[width=\linewidth]{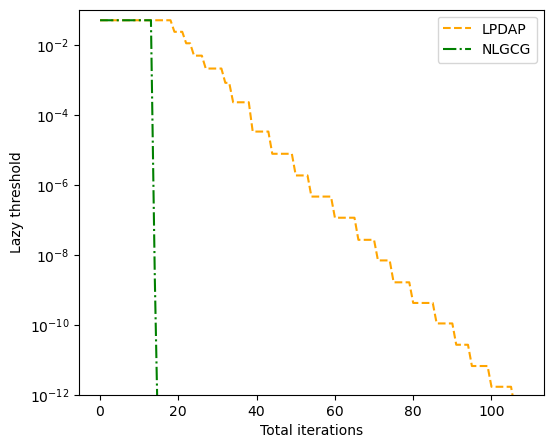}
  \caption{Evolution of the lazy threshold $\e_k$.}
  \label{fig: signal epsilon vs iterations}
\end{subfigure}
\hfill
\begin{subfigure}[b]{0.49\textwidth}
  \centering
  \includegraphics[width=\linewidth]{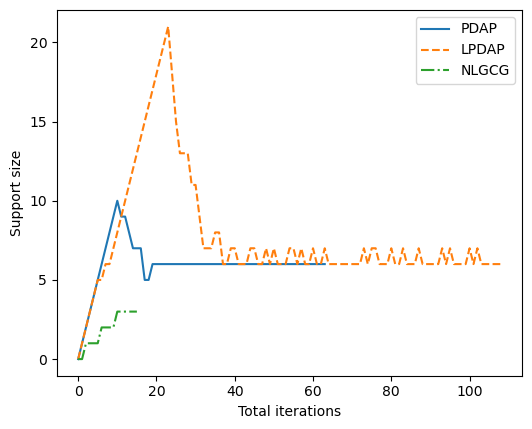}
  \caption{Support size in each iteration.}
  \label{fig: signal support in iterations}
\end{subfigure}
\caption{Lazy threshold $\e_k$ and support size for each algorithm.}
\end{figure}

\begin{figure}[ht]
    \centering
    \includegraphics[width=1\textwidth]{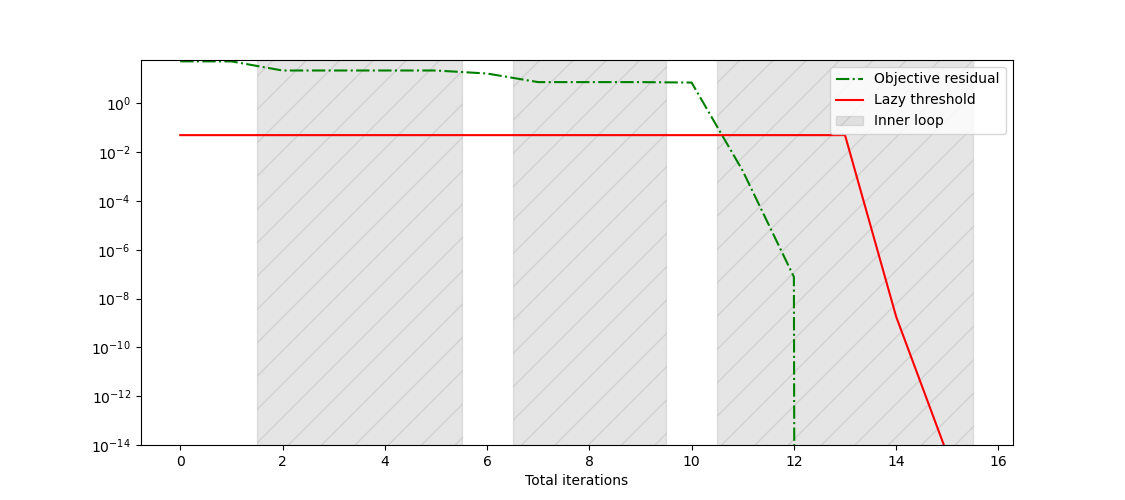}
    \caption{Progression of the Newton algorithm, with dashed areas corresponding to iterations within the inner loop.}
    \label{fig: signal local routine}
\end{figure}


Overall, both examples confirm our theoretical findings and highlight the potential of lazy updates in the considered setting.

\section*{Acknowledgments}
Both authors acknowledge funding by the Deutsche Forschungsgemeinschaft (DFG, German Research
Foundation) under Germany's Excellence Strategy – The Berlin Mathematics
Research Center MATH+ (EXC-2046/1, project ID: 390685689), sub-project AA4-14 ``Data-Driven Prediction of the Band-Gap for Perovskites''.

\bibliographystyle{plain}
\bibliography{sample}

\begin{thebibliography}{10}

\bibitem{Schiebinger}
Nicholas Boyd, Geoffrey Schiebinger, and Benjamin Recht.
\newblock The alternating descent conditional gradient method for sparse inverse problems.
\newblock {\em SIAM J. Optim.}, 27(2):616--639, 2017.

\bibitem{Chambolle}
Claire Boyer, Antonin Chambolle, Yohann De~Castro, Vincent Duval, Fr{\'e}d{\'e}ric De~Gournay, and Pierre Weiss.
\newblock On representer theorems and convex regularization.
\newblock {\em SIAM J. Optim.}, 29(2):1260--1281, 2019.

\bibitem{Braunblending}
G\'abor Braun, Sebastian Pokutta, Dan Tu, and Stephen Wright.
\newblock Blended conditional gradients: the unconditioning of conditional gradients.
\newblock In {\em Proceedings of ICML}, 2019.

\bibitem{BraunGábor2016LCGA}
G\'abor Braun, Sebastian Pokutta, and Daniel Zink.
\newblock Lazifying conditional gradient algorithms.
\newblock {\em J. Mach. Learn. Res.}, 20:Paper No. 71, 42, 2019.

\bibitem{Brediessparse}
Kristian Bredies and Marcello Carioni.
\newblock Sparsity of solutions for variational inverse problems with finite-dimensional data.
\newblock {\em Calc. Var. Partial Differential Equations}, 59(1):Paper No. 14, 26, 2020.

\bibitem{BrediesCarioniFanzonWalter}
Kristian Bredies, Marcello Carioni, Silvio Fanzon, and Daniel Walter.
\newblock Asymptotic linear convergence of fully-corrective generalized conditional gradient methods.
\newblock {\em Math. Program.}, 205(1-2):135--202, 2024.

\bibitem{BrediesKristian2009Agcg}
Kristian Bredies, Dirk~A. Lorenz, and Peter Maass.
\newblock A generalized conditional gradient method and its connection to an iterative shrinkage method.
\newblock {\em Comput. Optim. Appl.}, 42(2):173--193, 2009.

\bibitem{BrediesInverseroblems}
Kristian Bredies and Hanna~Katriina Pikkarainen.
\newblock Inverse problems in spaces of measures.
\newblock {\em ESAIM Control Optim. Calc. Var.}, 19(1):190--218, 2013.

\bibitem{ChizatLénaïc2022Soom}
L\'ena\"ic Chizat.
\newblock Sparse optimization on measures with over-parameterized gradient descent.
\newblock {\em Math. Program.}, 194(1-2):487--532, 2022.

\bibitem{DenoyelleQuentin2020TsFa}
Quentin Denoyelle, Vincent Duval, Gabriel Peyr\'e, and Emmanuel Soubies.
\newblock The sliding {F}rank-{W}olfe algorithm and its application to super-resolution microscopy.
\newblock {\em Inverse Problems}, 36(1):014001, 42, 2020.

\bibitem{mitsos17}
Hatim Djelassi and Alexander Mitsos.
\newblock A hybrid discretization algorithm with guaranteed feasibility for the global solution of semi-infinite programs.
\newblock {\em J. Global Optim.}, 68(2):227--253, 2017.

\bibitem{dunnopen}
J.~C. Dunn and S.~Harshbarger.
\newblock Conditional gradient algorithms with open loop step size rules.
\newblock {\em J. Math. Anal. Appl.}, 62(2):432--444, 1978.

\bibitem{Duvalsource}
Vincent Duval and Gabriel Peyr\'e.
\newblock Exact support recovery for sparse spikes deconvolution.
\newblock {\em Found. Comput. Math.}, 15(5):1315--1355, 2015.

\bibitem{Eftekhari}
Armin Eftekhari and Andrew Thompson.
\newblock Sparse inverse problems over measures: Equivalence of the conditional gradient and exchange methods.
\newblock {\em SIAM Journal on Optimization}, 29(2):1329--1349, 2019.

\bibitem{fedorov}
V.~V. Fedorov.
\newblock {\em Theory of optimal experiments}, volume No. 12 of {\em Probability and Mathematical Statistics}.
\newblock Academic Press, New York-London, 1972.

\bibitem{Flinth}
Axel Flinth, Fr\'ed\'eric de~Gournay, and Pierre Weiss.
\newblock On the linear convergence rates of exchange and continuous methods for total variation minimization.
\newblock {\em Math. Program.}, 190(1-2):221--257, 2021.

\bibitem{Flinth25}
Axel Flinth, Fr\'ed\'eric de~Gournay, and Pierre Weiss.
\newblock Grid is good. {A}daptive refinement algorithms for off-the-grid total variation minimization.
\newblock {\em Open J. Math. Optim.}, 6:Art. No. 3, 27, 2025.

\bibitem{hettich}
R.~Hettich and K.~O. Kortanek.
\newblock Semi-infinite programming: Theory, methods, and applications.
\newblock {\em SIAM Review}, 35(3):380--429, 1993.

\bibitem{Huynh}
Phuoc-Truong Huynh, Konstantin Pieper, and Daniel Walter.
\newblock Towards optimal sensor placement for inverse problems in spaces of measures.
\newblock {\em Inverse Problems}, 40(5):Paper No. 055007, 43, 2024.

\bibitem{pmlr-v28-jaggi13}
Martin Jaggi.
\newblock Revisiting {Frank-Wolfe}: Projection-free sparse convex optimization.
\newblock In Sanjoy Dasgupta and David McAllester, editors, {\em Proceedings of the 30th International Conference on Machine Learning}, volume~28 of {\em Proceedings of Machine Learning Research}, pages 427--435, Atlanta, Georgia, USA, 17--19 Jun 2013. PMLR.

\bibitem{KunischGCG}
Karl Kunisch and Daniel Walter.
\newblock On fast convergence rates for generalized conditional gradient methods with backtracking stepsize.
\newblock {\em Numer. Algebra Control Optim.}, 14(1):108--136, 2024.

\bibitem{juliensvm}
Simon Lacoste-Julien, Martin Jaggi, Mark Schmidt, and Patrick Pletscher.
\newblock Block-coordinate {Frank-Wolfe} optimization for structural {SVMs}.
\newblock In Sanjoy Dasgupta and David McAllester, editors, {\em Proceedings of the 30th International Conference on Machine Learning}, volume~28 of {\em Proceedings of Machine Learning Research}, pages 53--61, Atlanta, Georgia, USA, 17--19 Jun 2013. PMLR.

\bibitem{LanPokutta2017}
G.~Lan, Sebastian Pokutta, Y.~Zhou, and Daniel Zink.
\newblock Conditional accelerated lazy stochastic gradient descent.
\newblock In {\em Proceedings of the International Conference on Machine Learning (ICML)}, 2017.

\bibitem{Calatroni25}
Marta Lazzaretti, Claudio Estatico, Alejandro Melero, and Luca Calatroni.
\newblock Off-the-grid regularisation for {P}oisson inverse problems.
\newblock {\em Comput. Optim. Appl.}, 91(2):827--860, 2025.

\bibitem{LeykekhmanDmitriy2020Naos}
Dmitriy Leykekhman, Boris Vexler, and Daniel Walter.
\newblock Numerical analysis of sparse initial data identification for parabolic problems.
\newblock {\em ESAIM Math. Model. Numer. Anal.}, 54(4):1139--1180, 2020.

\bibitem{dissertationMilzarek}
Andre Milzarek.
\newblock {\em Numerical methods and second order theory for nonsmooth problems}.
\newblock PhD thesis, Technische Universität München, 2016.

\bibitem{Mine}
H.~Mine and M.~Fukushima.
\newblock A minimization method for the sum of a convex function and a continuously differentiable function.
\newblock {\em J. Optim. Theory Appl.}, 33(1):9--23, 1981.

\bibitem{oustry}
Antoine Oustry and Martina Cerulli.
\newblock Convex semi-infinite programming algorithms with inexact separation oracles.
\newblock {\em Optim. Lett.}, 19(3):437--462, 2025.

\bibitem{dissertationPieper}
Konstantin Pieper.
\newblock {\em Finite element discretization and efficient numerical solution of elliptic and parabolic sparse control problems}.
\newblock PhD thesis, Technische Universität München, 2015.

\bibitem{PieperKonstantin2021Lcoa}
Konstantin Pieper and Daniel Walter.
\newblock Linear convergence of accelerated conditional gradient algorithms in spaces of measures.
\newblock {\em ESAIM Control Optim. Calc. Var.}, 27:Paper No. 38, 37, 2021.

\bibitem{Tuomo1}
Tuomo Valkonen.
\newblock Proximal methods for point source localisation.
\newblock {\em J. Nonsmooth Anal. Optim.}, 4:Paper No. 10433, 36, 2023.

\bibitem{Tuomo2}
Tuomo Valkonen.
\newblock Point source localisation with unbalanced optimal transport, 2025.

\bibitem{wachsmuthwalter}
Gerd Wachsmuth and Daniel Walter.
\newblock No-gap second-order conditions for minimization problems in spaces of measures.
\newblock \url{https://arxiv.org/abs/2403.12001}, 2024.

\bibitem{dissertationWalter}
Daniel Walter.
\newblock {\em On sparse sensor placement for parameter identification problems with partial differential equations}.
\newblock PhD thesis, Technische Universität München, 2019.

\bibitem{wynn}
Henry~P. Wynn.
\newblock Results in the theory and construction of {$D$}-optimum experimental designs.
\newblock {\em J. Roy. Statist. Soc. Ser. B}, 34:133--147, 170--186, 1972.

\bibitem{yusparsity}
Yaoliang Yu, Xinhua Zhang, and Dale Schuurmans.
\newblock Generalized conditional gradient for sparse estimation.
\newblock {\em J. Mach. Learn. Res.}, 18:Paper No. 144, 46, 2017.

\end{thebibliography}

\appendix 

\section{Technical proofs}

\subsection{Proofs for Section \ref{sec:optimality}} \label{app:optimality}

\begin{proof}[Proof of Proposition~\ref{prop: sublevel general}]
    Consider the set $\mathcal{N}(\bar{y})$ from Assumption~\eqref{ass:strongconv}. We will first show that $Ku\in\mathcal{N}(\bar{y})$ for all $u$ with $r_J(u)$ small enough. If this were not the case, there would exist an $\epsilon>0$ and a sequence $(u_k)\subset\Mm$ with $r_J(u_k)\leq1/k$ and $\Vert Ku-\bar{y}\Vert_Y>\epsilon$ for all $k$. The weak*-compactness of the sublevel sets of $r_J$, given by Proposition~\ref{prop:existence}, would imply the existence of a weak*-convergent subsequence, also denoted by $(u_k)$ for readability. The weak* lower semicontinuity of $J$ and the uniqueness of $\bar{u}$ would then yield $u_k\rightharpoonup^*\bar{u}$ and, by the weak*-to-strong continuity of $K$, also $Ku_k\rightarrow\bar{y}$. This contradicts the assumption on $\epsilon$.

    The statements \eqref{prop: sublevel y and nebla F} and \eqref{prop: sublevel p convergence} then follow directly from \cite[Lemma 5.8]{PieperKonstantin2021Lcoa}.

    Now, we prove \eqref{prop: sublevel norm convergence}. On the one hand, it holds that
    \begin{align}
        \alpha\mnorm{\bar{u}}-\alpha\mnorm{u}&=\langle\bar{p} , \bar{u}\rangle-\alpha\mnorm{u}\leq\langle\bar{p} , \bar{u}-u\rangle\\
        &=-\langle\nabla F(K\bar{u}) , K\bar{u}-Ku\rangle\leq\Vert\nabla F(K\bar{u})\Vert_Y\sqrt{r_J(u)/\gamma}
    \end{align}
    for all $u\in\Mm$ with $r_J(u)$ small enough, where the last inequality uses \eqref{prop: sublevel y and nebla F}. On the other hand, we can use the convexity of $F$ to write
    \begin{align}
        r_J(u)\geq-\langle\bar{p} , u-\bar{u}\rangle+\alpha\mnorm{u}-\alpha\mnorm{\bar{u}}\geq-\Vert\nabla F(K\bar{u})\Vert_Y\sqrt{r_J(u)/\gamma}+\alpha\mnorm{u}-\alpha\mnorm{\bar{u}} .
    \end{align}
    Putting both directions together, we conclude that there exists a constant $c_\Mm>0$ such that 
    \begin{equation}
        \left\vert\mnorm{u}-\mnorm{\bar{u}}\right\vert\leq c_\Mm\sqrt{r_J(u)}
    \end{equation}
    for all $u\in\Mm$ with $r_J(u)$ small enough.

    To show \eqref{prop: sublevel coef convergence}, we use a contradiction argument similar to the one at the beginning of the proof. We construct a sequence $(u_k)$ with $\A_{u_k}\subset\Omega_{R'}$ and $r_J(u_k)\leq1/k$ such that for all $k$ there is a $j_k\leq \bar{N}$ with $\mu_k^{j_k}<\bar{\mu}$. Then a subsequence, also denoted by $(u_k)$, satisfies $u_k\rightharpoonup^*\bar{u}$. Since $\A_{u_k}\subset\Omega_{R'}$, for all $j\leq \bar{N}$ there exists a $\phi^j\in\cC$ such that $u_k(B_{R'}(\bar{x}^j))=\langle\phi^j , u_k\rangle\rightarrow\langle\phi^j , \bar{u}\rangle=\bar{\lambda}^j$. This contradicts the definition of $j_k$ and concludes the proof.
\end{proof}

\subsection{Proofs for Section \ref{section:lpdap}} \label{app:lpdap}

\begin{proof}[Proof of Lemma~\ref{lemma: drop valid}]
We start by setting $\tilde{u} \coloneqq u-u^\text{drop}$ and write $\tilde{u}=\tilde u^1 + \tilde{u}^2$, where
\begin{equation*}
    \mathcal{D}^1_u  \coloneqq \left\{ x \in \mathcal{D}_u\;|\; |p_u(x)| \leq \alpha -\sigma/2 \right\}, \quad \mathcal{D}^2_u \coloneqq \mathcal{D}_u \setminus \mathcal{D}^1_u, \quad \tilde{u}^1 \coloneqq u \mres \mathcal{D}^1_u, \quad  \tilde{u}^2 \coloneqq u \mres \mathcal{D}^2_u .
\end{equation*}
By definition, and noting that $\mnorm{u}=\mnorm{u^{\text{drop}}}+\mnorm{\tilde{u}^1}+\mnorm{\tilde{u}^2}$, we obtain
\begin{align}\label{eq: drop proof inequality}
    \alpha \mnorm{u}- \langle p_u, u \rangle  \geq \alpha \mnorm{u^\text{drop}}- \langle p_u, u^{\text{drop}} \rangle + \frac{\sigma}{2} \mnorm{\tilde{u}^1}+\alpha \mnorm{\tilde{u}^2} . 
\end{align}
We can estimate
\begin{equation}\label{eq: drop first estimate}
    \alpha \mnorm{u^{\text{drop}}}- \langle p_u, u^{\text{drop}} \rangle \geq \langle \bar{p}- p_u, u^{\text{drop}} \rangle , \quad |\langle \bar{p}- p_u, u^{\text{drop}} \rangle| \leq c_1 M\sqrt{r_J(u)} , 
\end{equation}
where the constant $c_1$ represents that from \eqref{prop: sublevel p convergence}. A similar estimate is possible on the left-hand side of \eqref{eq: drop proof inequality}. For this, notice that
\begin{align}
    \vert\langle p_u ,  u\rangle-\alpha\mnorm{\bar{u}}\vert&=\vert\langle p_u ,  u\rangle-\langle\bar{p} , \bar{u}\rangle\vert\leq\vert \langle p_u-\bar{p} ,  u\rangle\vert+\vert\langle\bar{p} , u-\bar{u}\rangle\vert\\
    &=\vert\langle p_u-\bar{p} ,  u\rangle\vert+\vert\langle\nabla F(K\bar{u}) , Ku-K\bar{u}\rangle\vert\\
    &\leq c_2\sqrt{r_J(u)} ,
\end{align}
where $c_2$ is some constant resulting from \eqref{prop: sublevel y and nebla F} and \eqref{prop: sublevel p convergence} for $r_J(u)$ small enough. We can use this to write
\begin{equation}\label{eq: drop second estimate}
    \alpha\mnorm{u}-\langle p_u , u\rangle=\alpha\mnorm{u}-\alpha\mnorm{\bar{u}}+\alpha\mnorm{\bar{u}}-\langle p_u ,  u\rangle\leq c_3\sqrt{r_J(u)} ,
\end{equation}
where $c_3$ results from combining $c_2$ with \eqref{prop: sublevel norm convergence}. By setting \eqref{eq: drop first estimate} and \eqref{eq: drop second estimate} in \eqref{eq: drop proof inequality}, we conclude that there exists a constant $c_4$ such that
\begin{equation}
    \mnorm{\tilde{u}^1_k}+ \mnorm{\tilde{u}^2_k}\leq c_4\sqrt{r_J(u)} .
\end{equation}
Finally, by Taylor expansion, we arrive at
\begin{align*}
        J(u^{\textup{drop}})-J(u) &\leq \langle p_u,\tilde{u} \rangle-\alpha\mnorm{\tilde u}+\frac{LC_K}{2}\mnorm{\tilde{u}}^2 \\ &\leq -\frac{\sigma}{2} \mnorm{\tilde{u}^1}-\alpha \mnorm{\tilde{u}^2}+ \frac{LC_K}{2}\left(\mnorm{\tilde{u}^1}+\mnorm{\tilde{u}^2} \right)^2 ,
\end{align*}
where the right-hand side is negative for $r_J(u)$ small enough. This shows the existence of a desired $\Delta\leq\Delta(R)$. In particular, if $r_J(u)\leq\Delta$, then \eqref{prop: p local sign} and \eqref{prop: p sigma} hold for both $u$ and $u^\text{drop}$. The latter, together with the construction of $\D_u$, implies that $\A_{u^\text{drop}}\subset\Omega_R$, while the former allows us to write
\begin{equation}
    \sign(u^\text{drop})=\sign(p_u)=\sign(p_{u^\text{drop}})\quad\text{on}\quad\Omega_R.
\end{equation}
The property $\A_{u^\textup{drop}}\cap B_R(\bar{x}^j)\neq\emptyset$ for all $j\leq\bar{N}$ follows from \eqref{prop: sublevel coef convergence} and $\A_{u^\text{drop}}\subset\Omega_R$.
\end{proof}

\begin{proof}[Proof of Lemma~\ref{lemma: coefficient same Phi}]
    It is clear that for all $j\leq N$ it holds $p_{w}^u(x^j)=\sign(u(\{x^j\}))p_{v^u_w}(x^j)$. Because of $\A_u\subset\Omega_R$, \eqref{prop: p local sign} implies that for all $w$ such that $J(v^u_w)\leq J(u)$ it holds 
    \begin{equation}
        \sign(u(\{x^j\}))=\sign(p_u(x^j))=\sign(p_{v^u_w}(x^j))
    \end{equation}
    and, as a consequence, $p_{w}^u(x^j)=\vert p_{v^u_w}(x^j)\vert$ for all $j\leq N$. Furthermore, it also holds $\langle p_w^u , w\rangle=\langle p_{v_w^u} , v_w^u\rangle$. Thus, inserting this into \eqref{eq: explicit Phi u} and using \eqref{eq:combinedgap} shows that $\Phi^u(w)=\Phi_{\A_u}(v^u_w)$ for all such $w$.  
\end{proof}

\begin{proof}[Proof of Lemma~\ref{lemma: output of coefficient step}]
    First, note that $u=v_{w_0}^u$. Moreover, Algorithm \ref{alg: low-dimensional problem} is well-defined since \eqref{eq: positive coefficient update} admits minimizers. Thus, given the output $w_+$, we have
    \begin{equation}
        J(v_{w_+}^{u})=J^{u}(w_+)\leq J^{u}(w_0)=J(v^{u}_{w_0})=J(u) .
    \end{equation}
    as well as
    \begin{equation*}
        \sign(p_{u}(x))=\sign(u(\{x\}))=\sign(u_+(\{x\})) \quad \text{for all} \quad x \in \A_{u_+}\subseteq\A_u,
    \end{equation*}
    where the first equality follows by assumption and the second by construction of $u_+$. If $r_J(u)$ is small enough, \eqref{prop: p local sign} applies both to $u$ and $u_+=v_{w+}^u$, from which we finally conclude $\sign(p_{u_+}(x))=\sign(u_+(\{x\}))$ for all $x \in \A_{u_+}$. Furthermore, it holds 
    \begin{equation}
        \Phi_{\A_{u_{k+}}}(u_+)\leq\Phi_{\A_u}(u_+)=\Phi_{\A_{u}}(v^{u}_{w+})=\Phi^{u}(w_+)\leq\Psi\,,
    \end{equation}
    where we use $\A_{u_+}\subseteq\A_u$ and Lemma~\ref{lemma: coefficient same Phi}.
\end{proof}

\begin{proof}[Proof of Lemma \ref{lemma: LSI valid}]
We first argue that Algorithm \ref{alg: local improver} is well-defined, i.e. it produces a nonempty set $\mathcal{B}_u$ and we have $\mathcal{B}_u \subset \Omega_R$. For this purpose, let $x \in \A_u$ be arbitrary but fixed. Let $j$ be the unique index such that $x\in B_R(\bar{x}^j)$. On the one hand, according to \eqref{prop: p local max}, $|p_u|$ admits a unique maximizer $\widehat{x}_u^j$ on $B_R(\bar{x}^j)$, which satisfies $\nabla p_u (\widehat{x}_u^j)=0$. Consequently, \eqref{eq: LSI suboptimality ineuqality} and \eqref{eq: LSI Phi inequality} hold trivially for $x_{\text{LSI}}= \widehat{x}_u^j$. Furthermore,
\begin{equation}
    \Vert p_u-\bar{p}\Vert_{\cC(B_R(\bar{x}^j))}\leq c\sqrt{r_J(u)}
\end{equation}
for some $c$ from \eqref{prop: sublevel p convergence}, together with $\Vert p_u\Vert_{\cC(B_R(\bar{x}^j))}=\vert p_u(\widehat{x}_u^j)\vert$ and $\Vert\bar{p}\Vert_{\cC(B_R(\bar{x}^j))}=\alpha$, implies that \eqref{eq: LSI sigma} is also satisfied by this choice of $x_{\textup{LSI}}$ if $r_J(u)$ is small enough. On the other hand, for any point $x_{\text{LSI}} \in B_{2R}(x)$ with \eqref{eq: LSI sigma} we must have $x_{\text{LSI}} \in (B_{2R}(x) \cap \Omega_R) = B_R(\bar{x}^j)$ by \eqref{prop: p sigma}.

It remains to show that $\mathcal{B}_u$ contains exactly $\bar{N}$ points, one per ball $B_R(\bar{x}^j)$. We first observe that for any $x \in \A_u \cap B_R(\bar{x}^j)$ we have
\begin{equation*}
    \A_u \setminus B_{2R}(x)= \A_u \cap (\Omega_R \setminus B_R(\bar{x}^j)).
\end{equation*}
At the same time, \eqref{prop: sublevel coef convergence}, combined with $\A_u\subset\Omega_R$, implies $\A_u \cap B_R(\bar{x}^j) \neq \emptyset$ for all $j \in \{1, \dots,\bar{N}\}$. Combining both observations yields the desired statement.
\end{proof}

\begin{proof}[Proof of Lemma \ref{lemma: LSI suboptimal}]
Let $x \in  \A_u\cap B_R(\bar{x}^{\bar \jmath_u} )$ be arbitrary but fixed. Then there holds
    \begin{align*}
        \vert p_u(\widehat{x}_u)\vert-\vert p_u(x)\vert&=\vert p_u(\widehat{x}_u)\vert-\vert p_u(x^{u,\bar \jmath_u}_{\text{LSI}})\vert+\vert p_u(x^{u,\bar \jmath_u}_{\text{LSI}})\vert-\vert p_u(x)\vert\\
        &\leq2R\Vert\nabla p_u(x^{u,\bar \jmath_u}_{\text{LSI}})\Vert+\vert p_u(x^{u,\bar \jmath_u}_{\text{LSI}})\vert-\vert p_u(x)\vert\\
        &\leq 2\left(\vert p_u(x^{u,\bar \jmath_u}_{\text{LSI}})\vert-\vert p_u(x)\vert\right),
    \end{align*}
    where the first inequality follows from \eqref{prop: p local max} and the second is due to \eqref{eq: LSI suboptimality ineuqality}. We conclude by noting that 
    \begin{equation*}
        \vert p_u(\widehat{x}^u_{\text{LSI}})\vert-\max_{x\in\A_u\cap B_{2R}(\widehat{x}^u_{\text{LSI}})}\vert p_u(x)\vert \geq \vert p_u(x^{u,\bar \jmath_u}_{\text{LSI}})\vert-\max_{x\in\A_u\cap B_{2R}(x^{u,\bar \jmath_u}_{\text{LSI}})}\vert p_u(x)\vert \\ 
    \end{equation*}
    as well as
    \begin{align*}
        \vert p_u(x^{u,\bar \jmath_u}_{\text{LSI}})\vert-\max_{x\in\A_u\cap B_{2R}(x^{u,\bar \jmath_u}_{\text{LSI}})}\vert p_u(x)\vert&=\vert p_u(x^{u,\bar \jmath_u}_{\text{LSI}})\vert-\max_{x\in\A_u\cap B_{R}(\bar{x}^{\bar \jmath})}\vert p_u(x)\vert\\& \geq \frac{1}{2}\left(\vert p_u(\widehat{x}_u)\vert-\max_{x\in\A_u\cap B_R(\bar{x}^j)}\vert p_u(x)\vert\right).
    \end{align*}
\end{proof}

\begin{proof}[Proof of Lemma~\ref{lemma:bound on K}]
From $\A_u\cup\mathcal{B}_u\subset\Omega_R$ it follows
\begin{equation*}
    \langle \phi , \tilde{v}_u-u  \rangle= \sum^{\bar{N}}_{j=1} \sum_{x \in \A_u\cap B_R(\bar{x}^j)  }  u(\{x\}) \left( \phi(x^{u,j}_{\text{LSI}})-\phi(x) \right)
\end{equation*}
for all $\phi\in\cC$, as well as
\begin{equation*}
    \Vert K(\tilde{v}_u-u)\Vert_Y= \sup_{\|y\|_Y=1} (y, K(\tilde{v}_u-u) )_Y=\sup_{\|y\|_Y=1} \langle K_* y ,  \tilde{v}_u-u  \rangle .
\end{equation*}
Consequently, we obtain
\begin{equation*}
    \Vert K(\tilde{v}_u-u)\Vert_Y \leq C_{K'} \sum^{\bar{N}}_{j=1} \sum_{x \in \A_u\cap B_R(\bar{x}^j)  }  |u(\{x\})| \|x^{u,j}_{\text{LSI}}-x\| .
\end{equation*}
    For every $x \in \A_u\cap B_{R}(\bar{x}^j) $, we estimate
    \begin{equation}
        \|x^{u,j}_{\text{LSI}}-x\|\leq\| x^{u,j}_{\text{LSI}}-\widehat{x}_u^j\|+\|\widehat{x}_u^j-\bar{x}^j\|+\|\bar{x}^j-x\Vert .
    \end{equation}
    
    For the first term, \eqref{prop: p curvature and growth}, \eqref{prop: p local max}, and \eqref{eq: LSI Phi inequality} yield
    \begin{equation}
        \Vert x^{u,j}_{\text{LSI}}-\widehat{x}_u^j\Vert\leq 4 \sqrt{\frac{R}{\theta}\Phi_{\A_u}(u)}.
    \end{equation}
    
    For the second term, we argue along the lines of \cite[Lemma 5.14]{PieperKonstantin2021Lcoa}, combined with \eqref{prop: sublevel p convergence}, to obtain
    \begin{equation}
        \Vert\widehat{x}_u^j-\bar{x}^j\Vert\leq \frac{4 C_{K'}L}{\theta\sqrt{\gamma}}\sqrt{\Phi(u)}.
    \end{equation}

    Finally, the third term can be treated analogously to \cite[Proposition~6.8]{BrediesCarioniFanzonWalter}, leading to
    \begin{equation}
       \sum^{\bar{N}}_{j=1} \sum_{x\in \A_u\cap B_R(\bar{x}^j)}| u(\{x\})| \| x-\bar{x}^j\|\leq 2 \sqrt{\frac{M}{\theta}\Phi(u)} .
    \end{equation}

Combining these estimates with $\Phi_{\A_u}(u) \leq \Phi(u)$ yields the desired statement.
\end{proof}

\begin{proof}[Proof of Corollary \ref{corollary: linear lpdap}]
    Using $0<\zeta<1$, we can see that $k+s\geq3\log_\zeta(\epsilon)$ is equivalent to $\zeta^{\frac{1}{3}(k+s)}\leq\epsilon$. From the proof of Theorem~\ref{theorem:Big O linear} we can conclude that, given an $\epsilon$ small enough, all pairs $(k,s)$ satisfying $k+s\geq 3\log_\zeta(\epsilon)$ imply $r_J(u_{k,s})\leq\epsilon$. If $\tilde{k}$ is such that this holds for all $\epsilon\leq\zeta^{\frac{1}{3}\tilde{k}}$, then, for any $k+s\geq\tilde{k}$ it holds that $k+s=3\log_\zeta(\zeta^{\frac{1}{3}(k+s)})$, $\zeta^{\frac{1}{3}(k+s)}\leq\zeta^{\frac{1}{3}\tilde{k}}$, and 
    \begin{equation}
        r_J(u_{k,s})\leq \zeta^{\frac{1}{3}(k+s)} .
    \end{equation}
\end{proof}

\subsection{Proofs for Section \ref{section: Newton}} \label{app: Newton}

\begin{proof}[Proof of Lemma~\ref{lem:genericsparse}]
    According to Lemma~\ref{lemma: drop valid} and by the assumptions on $\tilde{u}_n$, we have
\begin{equation*}
    \A_{\tilde{u}_n} \subset \A_{u^{\text{drop}}_n} \subset\Omega_R , \quad \sign(\tilde{u}_n(\{x\}))=\sign(u^{\text{drop}}_n(\{x\}))=\sign(p_{u^{\text{drop}}_n}(x))=\bar{\lambda}^j=\sign(p_{\tilde{u}_n}(x))
\end{equation*}
for all $x \in \A_{\tilde{u}_n}$ and all $n$ large enough. Furthermore, $J(\tilde{u}_n)\leq J(u_n)$ implies $\tilde{u}_n\rightharpoonup^*\bar{u}$.  By construction of the local merging step as well as by choice of $R$, we thus have
\begin{equation*}
    u^{\text{lump}}_n= \sum^{\bar{N}}_{j=1} \lambda^j_n \delta_{x^j_n}, \quad \text{where} \quad x^j_n \in \argmax_{x\in\A_{\tilde{u}_n} \cap B_R(\bar{x}^j)}\vert p_{\tilde{u}_n}(x)\vert, \quad \lambda^j_n=\tilde{u}_n(B_R(\bar{x}^j)).
\end{equation*}
Set $\bar z^{\text{lump}}_n=(\mathbf{x}_n, \lambda_n)$, where $\mathbf{x}_n=(x^1_n, \dots, x^{\bar{N}}_{n})$ and $\lambda_n=(\lambda^1_n, \dots, \lambda^{\bar N}_n)$. Note that $\bar z^{\text{lump}}_n$ is a minimal representer of $u^{\text{lump}}_n$ and
\begin{equation*}
    \operatorname{dist}(\bar z^{\text{lump}}_n, \bar{\mathcal{Z}})=\operatorname{dist}(z^{\text{lump}}_n, \bar{\mathcal{Z}}).
\end{equation*}
Hence, it suffices to show that $\bar z^{\text{lump}}_n \rightarrow \bar{z}$. For this purpose, we immediately get $\lambda^j_n \rightarrow \bar{\lambda}^j$ due to $\tilde{u}_n  \weakstar \bar{u}$. Next, we note that
\begin{align*}
  0&=  \lim_{n \rightarrow \infty} \left\lbrack  \alpha \mnorm{\tilde{u}_n}- \dual{p_{\tilde{u}_n}}{\tilde{u}_n} \right \rbrack \\ &\geq \lim_{n \rightarrow \infty} \left\lbrack\alpha \mnorm{u^{\text{lump}}_n}- \dual{p_{\tilde{u}_n}}{u^{\text{lump}}_n}\right \rbrack =  \lim_{n \rightarrow \infty} \left \lbrack \alpha \mnorm{u^{\text{lump}}_n}- \dual{\bar{p}}{u^{\text{lump}}_n} \right \rbrack \geq 0,
\end{align*}
where the first inequality follows from $\mnorm{\tilde{u}_n}=\mnorm{u^{\text{lump}}_n}$ as well as
\begin{equation*}
    \dual{p_{\tilde{u}_n}}{\tilde{u}_n}= \sum^{\bar{N}}_{j=1} \sum_{x \in \A_{\tilde{u}_n} \cap B_{R}(\bar{x}^j)} |\tilde{u}_n(\{x\})| |p_{\tilde{u}_n}(x)| \leq \sum^{\bar{N}}_{j=1} |\lambda^j_n| |p_{\tilde{u}_n}(x^n_j)|=\dual{p_{\tilde{u}_n}}{u^{\text{lump}}_n}
\end{equation*}
due to Lemma \ref{lemma: drop valid} as well as by choice of $x^j_n$. Rearranging this further, we obtain 
\begin{align*}
0=\lim_{n \rightarrow \infty} \left \lbrack \alpha \mnorm{u^{\text{lump}}_n}- \dual{\bar{p}}{u^{\text{lump}}_n} \right \rbrack &= \lim_{n \rightarrow \infty} \left\lbrack\sum^{\bar{N}}_{j=1} \lambda^j_n (\alpha -|\bar{p}(x^j_n)|) \right \rbrack =\lim_{n \rightarrow \infty} \left\lbrack\sum^{\bar{N}}_{j=1} \bar\lambda^j (\alpha -|\bar{p}(x^j_n)|) \right \rbrack
\end{align*}
from which we conclude $|\bar{p}(x^j_n)| \rightarrow \alpha$. Since $x^j_n \in B_R (\bar{x}_j)$ for all $n$ large enough and $|\bar{p}(x)|< \alpha$ for all $x\in\Omega\,\backslash\,\bar{\A}$, we get $x^j_n \rightarrow \bar{x}^j$ for all $j\leq \bar{N}$.
\end{proof}

\begin{proof}[Proof of Lemma~\ref{lem:insideofball}]
    According to \eqref{eq:propNewton} and \eqref{eq:quadconvintro}, we have $z^{\text{New}}_{k,s} \in \mathring{\mathcal{Z}}^{\bar N}$ as well as $\operatorname{dist}(z^{\text{New}}_{k,s}, \bar{\mathcal{Z}})< \nu_0$ and
\begin{equation*}
    r_{J}(u^{\text{New}}_{k,s}) \leq C_{\text{New}} \, r_{J}(u_{k,s})^2.
\end{equation*}
Noting that $z^{\text{New}}_{k,s}=\textup{\texttt{MR}}(u^{\text{New}}_{k,s})$, it thus suffices to show that $u_{k,s+1}$ is well-defined and it holds $u_{k,s+1}=u^{\text{New}}_{k,s}$. Regarding the first statement, we point out that $z_{k,s}$ and $z^{\text{New}}_{k,s}$ satisfy \eqref{eq:acceptanceNewton1} and \eqref{eq:acceptanceNewton2} due to \eqref{eq:descentprop1} and \eqref{eq:NewtonconnectoGCG}, respectively. Hence $u_{k,s+1}$ is well-defined. According to \eqref{eq:Newtondrop}, we further conclude $u^{\text{New}}_{k,s}=\textup{\texttt{DropStep}}(u^{\text{New}}_{k,s})$ and, finally, $u^{\text{New}}_{k,s}=\texttt{LM}(u^{\text{New}}_{k,s},R)$ since   
\begin{equation*}
   \A_{u^{\text{New}}_{k,s}} \subset \Omega_R, \quad \A_{u^{\text{New}}_{k,s}} \cap B_R(\bar{x}_j) \neq \emptyset \quad \text{for all} \quad j \leq \bar{N}.
\end{equation*}
Summarizing these observations, we obtain $u_{k,s+1}=u^{\text{New}}_{k,s}$.
\end{proof}

\end{document}